\documentclass[12pt]{amsart}

\usepackage{amsmath}
\usepackage{amssymb}
\usepackage{latexsym}
\usepackage{mathrsfs}
\usepackage{amsthm}
\usepackage{stmaryrd}
\usepackage[all]{xy}
\usepackage{xcolor}

\numberwithin{equation}{section}
\newtheorem{theorem}{\bf Theorem}[section]

\newtheorem{lemma}[theorem]{\bf Lemma}
\newtheorem{proposition}[theorem]{\bf Proposition}

\theoremstyle{definition}

\newtheorem{remark}[theorem]{\bf Remark}

\newtheorem{definition}[theorem]{\bf Definition}

\numberwithin{equation}{section}
\makeatletter
\let\c@theorem\c@equation
\makeatother

\newtheorem*{namedtheorem}{\theoremname}
\newcommand{\theoremname}{testing}

\renewcommand{\leq}{\leqslant}
\renewcommand{\geq}{\geqslant}
\newcommand{\norm}{\trianglelefteq}

\newcommand{\gen}[1]{\langle #1 \rangle}
\newcommand{\ol}{\overline}

\newcommand{\bs}{\backslash}

\newcommand{\lra}{\longrightarrow}

\newcommand{\N}{\mathfrak{N}}

\newcommand{\sS}{\mathscr{S}}

\newcommand{\A}{\mathcal{A}}
\newcommand{\B}{\mathcal{B}}
\newcommand{\C}{\mathcal{C}}

\newcommand{\F}{\mathcal{F}}

\renewcommand{\L}{\mathcal{L}}

\renewcommand{\O}{\mathcal{O}}
\newcommand{\R}{\mathcal{R}}
\newcommand{\T}{\mathcal{T}}
\newcommand{\Q}{\mathcal{Q}}

\newcommand{\W}{\mathcal{W}}

\newcommand{\Y}{\mathcal{Y}}
\newcommand{\Z}{\mathcal{Z}}
\renewcommand{\phi}{\varphi}

\DeclareMathAlphabet\EuR{U}{eur}{m}{n}
\SetMathAlphabet\EuR{bold}{U}{eur}{b}{n}
\newcommand{\curs}{\EuR}
\newcommand{\Ab}{\curs{Ab}}

\newcommand{\id}{\operatorname{id}}

\newcommand{\Hom}{\operatorname{Hom}}
\newcommand{\Mor}{\operatorname{Mor}}

\newcommand{\Aut}{\operatorname{Aut}}
\newcommand{\Out}{\operatorname{Out}}
\newcommand{\Inn}{\operatorname{Inn}}
\newcommand{\End}{\operatorname{End}}

\newcommand{\Syl}{\operatorname{Syl}}

\begin{document}
\title[Control of fixed points]{Control of fixed points and existence and
uniqueness of centric linking systems}
\author{George Glauberman}
\address{Department of Mathematics \\ University of Chicago \\ 5734 S.
University Ave \\ Chicago, IL 60637}
\email{gg@math.uchicago.edu}
\author{Justin Lynd}
\address{Department of Mathematical Sciences\\ University of Montana \\ 32
Campus Drive \\ Missoula, MT  59812} 
\email{justin.lynd@umontana.edu}
\date{\today}
\thanks{The second author was partially supported by NSA Young Investigator
Grant H98230-14-1-0312, and was supported by an AMS-Simons travel grant which
allowed for travel related to this work.}

\begin{abstract}
A. Chermak has recently proved that to each saturated fusion system over a
finite $p$-group, there is a unique associated centric linking system. B.
Oliver extended Chermak's proof by showing that all the higher cohomological
obstruction groups relevant to unique existence of centric linking systems
vanish. Both proofs indirectly assume the classification of finite simple
groups. We show how to remove this assumption, thereby giving a
classification-free proof of the Martino-Priddy conjecture concerning the
$p$-completed classifying spaces of finite groups.  Our main tool is a 1971
result of the first author on control of fixed points by $p$-local subgroups.
This result is directly applicable for odd primes, and we show how a slight
variation of it allows applications for $p=2$ in the presence of offenders.
\end{abstract}

\maketitle

\section{Introduction}

Let $p$ be a prime and let $S$ be a Sylow $p$-subgroup a finite group $G$ whose
order is divisible by $p$. The pattern of $G$-conjugacy of the subgroups of
$S$, sometimes called the $p$-\emph{fusion} in $S$ induced by $G$, exerts
considerable influence on the structure and the invariants of $G$ around the
prime $p$. 
For example, it determines the mod-$p$ cohomology of $G$ (by the Stable
Elements Theorem of Cartan and Eilenberg \cite[XII.10.1]{CartanEilenberg1956}),
and the homomorphic images of $G$ that are $p$-groups (by the Hyperfocal
Subgroup Theorem of Puig \cite[Theorem~1.33]{CravenTheory}). Also, while
$p$-fusion has a strong influence on the representation theory of $G$ in
characteristic $p$, the precise connection remains mysterious despite being the
subject of several major conjectures in that area. One of the consequences of
the results in this paper is a new proof of the fact that $p$-fusion in a
finite group determines the homotopy type of the Bousfield-Kan $p$-completion
of the classifying space of the group. 



A \emph{saturated fusion system} over a finite $p$-group $S$ is a category with
objects the subgroups of $S$, and with morphisms injective group homomorphisms
between subgroups, subject to certain axioms originally formulated by Puig
\cite{Puig2006, Puig2009} and recalled in Section~\ref{S:limits}. Given a
finite group $G$ containing $S$ as a Sylow subgroup, there is a fusion system
$\F_S(G)$ over $S$ in which the morphisms are the conjugation maps between
subgroups induced by elements of the group $G$. The fusion system of a finite
group constitutes the standard example of a saturated fusion system.  However,
to any $p$-block of a finite group algebra in characteristic $p$ is associated
a fusion system codifying the $p$-local structure of the block, and the
commonality in these two situations was one of the main motivations behind
Puig's formalization. 

The ``mod-$p$ homotopy theory'' of the group $G$ is embodied in the
Bousfield-Kan $p$-completion of its classifying space. For most classes of
spaces, the task of determining homotopy type and mapping spaces after
$p$-completion is difficult. The Martino-Priddy conjecture asserts that for the
classifying space of a finite group, however, the homotopy type of its
$p$-completion is determined uniquely by the finite category $\F_S(G)$.  More
precisely, two finite groups have homotopy equivalent $p$-completed classifying
spaces if and only if there is an isomorphism between their Sylow $p$-subgroups
inducing an equivalence of categories between the corresponding fusion systems
\cite[Theorem~III.1.17]{AschbacherKessarOliver2011}.  The ``only if'' part of
this conjecture was proved by Martino-Priddy \cite{MartinoPriddy1996}.  The
``if'' part was proved by B.  Oliver \cite{Oliver2004, Oliver2006} using the
classification of the finite simple groups. 

Around the turn of the century, Broto, Levi, and Oliver initiated a detailed
study of Puig's categories from this homotopy-theoretic point of view
\cite{BrotoLeviOliver2003b, BrotoLeviOliver2003}. 
A \emph{centric linking system} is a certain extension category of a fusion
system $\F$ with a restricted collection of objects (the $\F$-\emph{centric}
subgroups, defined in Section~\ref{S:limits}). A canonical centric linking
system at the prime $p$ is easily constructed from an ambient group $G$. 
Moreover, Broto, Levi, and Oliver showed that two finite groups have homotopy
equivalent $p$-completed classifying spaces if and only if there is an
equivalence of categories between their centric linking systems. Thus, the
``if'' part of the Martino-Priddy conjecture is equivalent to the statement
that the fusion system of a finite group uniquely determines its centric
linking system. This focused attention on the question of whether each
saturated fusion system has a unique associated centric linking system, even
when $\F$ is not of the form $\F_S(G)$ for any finite group $G$.  Examples of
``exotic'' classifying spaces (like those that would arise out of such a
development) were first constructed by Benson \cite{Benson1998c}. 

A. Chermak has settled affirmatively the unique existence of centric linking
systems by direct construction \cite{Chermak2013}.  His construction is made
possible by a delicate filtration of the collection of $\F$-centric subgroups,
which makes use of the Thompson subgroup in a critical way, together with an
iterative procedure for extending a linking system on a given collection of
subgroups of $S$ to a linking system on a larger collection.  Chermak's
filtration allows a number of very useful reductions. In particular, within his
work the problem of building inductively a linking system on a larger set of
objects boils down to the problem of extending an automorphism of a linking
system of a $p$-constrained finite group (as defined below) to an automorphism
of the group.  It is in this residual situation where an indirect appeal to the
classification of finite simple groups was needed, in the form of the General
FF-module Theorem of Meierfrankenfeld and Stellmacher
\cite{MeierfrankenfeldStellmacher2012}. 

In a companion paper \cite{Oliver2013}, B. Oliver showed how to interpret
Chermak's proof in terms of the established Broto-Levi-Oliver obstruction
theory for the existence and uniqueness of linking systems \cite[\S
3]{BrotoLeviOliver2003}.  The obstruction groups appearing there are the higher
derived limits over the orbit category of the fusion system, at the level of
$\F$-centric subgroups, of the (contravariant) center functor $\Z_\F\colon
\O(\F^c) \to \Ab$ which sends an $\F$-centric subgroup to its center.  From the
point of view of this obstruction theory, Chermak's filtration gives a way to
filter the center functor $\Z_\F$ so that the higher limits of each subquotient
functor in the filtration can be shown to vanish.  Within Oliver's proof, the
problem is again reduced to the case where $\F$ is the fusion system of a
$p$-constrained finite group, and to showing that $\varprojlim{ }^1$ (when $p$
is odd) and $\varprojlim{ }^2$ (when $p=2$) of certain explicit subquotient
functors of the center functor on the orbit category of this group vanish. For
this, an appeal to the General FF-module Theorem gives a list of the possible
groups, and then the proof is finished by examining these cases.

In this paper, we study the residual situation in Oliver's version of Chermak's
Theorem and give a proof of Proposition~3.3 of \cite{Oliver2013} that does not
depend on the classification of finite simple groups.  When taken together with
the reduction via Chermak's filtration to this situation in \cite{Oliver2013},
a classification-free proof of existence and uniqueness of centric linking
systems, and thus also of the Martino-Priddy conjecture, is obtained at all
primes. 

\begin{theorem}[Oliver \cite{Oliver2013}]\label{T:main}
Let $\F$ be a saturated fusion system over a finite $p$-group. Then
$\varprojlim^k_{\O(\F^c)} \Z_\F = 0$ for all $k \geq 1$ if $p$ is odd, and for
all $k \geq 2$ if $p=2$.
\end{theorem}
\begin{proof}
When $p$ is odd, this follows from the proof of \cite[Theorem~3.4]{Oliver2013}
and Proposition~\ref{P:3.3odd} below. When $p=2$, it follows from the proof of
\cite[Theorem~3.4]{Oliver2013} and Proposition~\ref{P:3.3} below. 
\end{proof}

It was known very early that $\varprojlim^1_{\O(\F^c)} \Z_\F$ can be
nonvanishing when $p=2$. An example of this is given by the fusion system $\F$
of the alternating group $A_6$ at the prime $2$, where
$\varprojlim^1_{\O(\F^c)} \Z_\F$ is of order $2$ \cite[Proposition~1.6,
Ch.10]{Oliver2006}. 

\begin{theorem}[Chermak \cite{Chermak2013}]\label{C:main}
Each saturated fusion system has an associated centric linking system that is
unique up to isomorphism. 
\end{theorem}
\begin{proof} 
This follows from Theorem~\ref{T:main} together with
\cite[Proposition~3.1]{BrotoLeviOliver2003}.
\end{proof}


In addition to relying heavily on the reductions of Chermak and Oliver, our
arguments use variations on techniques of the first author for studying when,
for a finite group $G$ acting on an abelian $p$-group $D$, some subgroup $H$
\emph{controls fixed points} of $G$ on $D$ -- that is, when $C_D(H) = C_D(G)$.
In particular, very general conditions were given in
\cite[Theorem~A1.4]{Glauberman1971} under which this holds for a suitable
$p$-\emph{local subgroup} $H$ of $G$ (i.e., the normalizer of a nonidentity
$p$-subgroup). This general result is the basis for the statement, also found
in \cite{Glauberman1971}, that the normalizer of the Thompson subgroup
``controls weak closure of elements'' when $p$ is odd. We refer to \S14 of
\cite{Glauberman1971} for more details on this relationship.

For the purposes of this introduction, a finite group $\Gamma$ is said to be
$p$-\emph{constrained} if there is a normal $p$-subgroup $Y$ of $\Gamma$ such
that $C_\Gamma(Y) \leq Y$. (The usual definition of a $p$-constrained group is
more general \cite[p.268]{Gorenstein1980}.) The structure of such a group is
heavily influenced by the faithful action of $\Gamma/C_\Gamma(Z(Y))$ on the
abelian $p$-group $Z(Y)$.  

In order to explain how control of fixed points is helpful in computing limits
over orbit categories of $p$-constrained groups, we fix a $p$-constrained group
$\Gamma$ with normal $p$-subgroup $Y$ as above. Let $S$ be a Sylow $p$-subgroup
of $\Gamma$, set $\F = \F_S(\Gamma)$, and denote by $\sS(S)_{\geq Y}$ the
collection of all subgroups of $S$ containing $Y$ (each of which is
$\F$-centric). The filtrations of the center functor that feature in Chermak's
proof correspond to objectwise filtrations of the collection of $\F$-centric
subgroups. A collection $\Q \subseteq \sS(S)_{\geq Y}$ that is invariant under
$\F$-conjugacy and closed under passing to overgroups corresponds to a quotient
$\Z^\Q_\F$ of the center functor. For such a collection, one can form the
\emph{locality} $\L_\Q(\Gamma)$ (in the sense of \cite[2.10]{Chermak2013})
consisting of those $g \in \Gamma$ that conjugate a subgroup in the collection
$\Q$ to another subgroup in $\Q$.  Now with $\R = \sS(S)_{\geq Y} - \Q$,
an examination of the first part of the long exact sequence corresponding to
the short exact sequence $0 \to \Z_\F^\R \to \Z_\F^{\sS(S)_{\geq Y}} \to
\Z_\F^\Q \to 0$ of functors shows that
$\varprojlim{}^1 \Z_\F^\R = 0$ provided that the inclusion 
\[
\varprojlim{\!}^0 \Z_\F^{\sS(S)_{\geq Y}} \cong C_{Z(Y)}(\Gamma)
\hookrightarrow C_{Z(Y)}(\L_\Q(\Gamma)) \cong \varprojlim{\!}^0 \Z_\F^\Q
\]
is an isomorphism (see Lemma~\ref{L:oliexact} for more details). Hence
$\varprojlim{ }^1 \Z_\F^\R = 0$ provided some normalizer in $\Gamma$ of a
subgroup in the collection $\Q$ controls the fixed points of $\Gamma$ on
$Z(Y)$. This observation essentially completes the picture when $p$ is an odd
prime, as is shown in Section~\ref{S:podd}. 

For $p=2$, the canonical obstruction to having control of fixed points by a
$2$-local subgroup is the symmetric group $G = S_3$ acting on $D = C_2 \times
C_2$. Here the $2$-local subgroups are of order $2$ and have nontrivial fixed
points on $D$, while $G$ does not. The aim for $p=2$ is to isolate this
obstruction in the case where $D$ is an FF-module for $G$ -- that is, when $G$
has nontrivial \emph{offenders} on $D$ (Definition~\ref{D:offender} below). We
define the notion of a \emph{solitary offender}, which is an offender of order
$2$ having a specific type of embedding in $G$ with respect to a Sylow $2$-subgroup.
The standard example of a solitary offender is generated by a transposition in
an odd degree symmetric group acting on a natural module over the field with
two elements, where the transposition is, in particular, a transvection.  Away
from the obstruction posed by solitary offenders, control of fixed points by
the normalizer of a nonidentity normal -- but perhaps not characteristic --
subgroup of a Sylow $2$-subgroup is obtained in Section~\ref{S:norm2}. 

\begin{theorem}\label{T:solobs}
Let $G$ be a finite group, $S$ a Sylow $2$-subgroup of $G$, and $D$ an abelian
$2$-group on which $G$ acts faithfully. Assume that $G$ has a minimal
nontrivial offender on $D$ that is not solitary. Then there is a subgroup $J
\leq S$, generated by offenders and weakly closed in $S$ with respect to $G$,
such that $C_D(N_G(J)) = C_D(G)$. 
\end{theorem}
\begin{proof}
This follows from Theorem~\ref{T:glawc2} via Lemmas~\ref{L:controlbig} and
\ref{L:nonS3basednorm}.
\end{proof}

This could be viewed as the main result of Section~\ref{S:norm2}, although the
more detailed information contained in Lemmas~\ref{L:controlbig} and
\ref{L:nonS3basednorm} is needed for the sequel. When interpreted,
Theorem~\ref{T:solobs} gives the vanishing of $\varprojlim{\!}^1 \Z^\R_\F$ for
$\F$ the fusion system of any $2$-constrained group $\Gamma$ as above with $D =
Z(Y)$ and $G = \Gamma/C_\Gamma(D)$ having the prescribed action on $D$, and
with $\R$ consisting of those subgroups containing $Y$ whose images in $G$ do
not contain offenders on $D$.

The results of Section~\ref{S:norm2} are then applied to upgrade the vanishing
of $\varprojlim{\!}^1$ to that of $\varprojlim{\!}^2$ away from the canonical
obstruction, and this is done in Theorem~\ref{T:ntlim}. For this and for the
remaining arguments, we work with the bar resolution for these limits.  As a
result, the arguments involve questions about realizing an automorphism of a
locality as an inner automorphism of a group, and thus begin to resemble those
appearing in Chermak's work \cite{Chermak2013}. We hope that the preliminary
lemmas of Section~\ref{S:reduction2} will make more clear this connection for
those who are more familiar with Chermak's group-theoretic approach.

With a little more work, it is seen that in a minimal counterexample, $G$ is
generated by its solitary offenders on $D$, and in particular by transvections
on $\Omega_1(D)$.  In this paper, by a \emph{natural module} for a symmetric
group $S_m$ ($m \geq 3$), we mean the lone nontrivial irreducible composition
factor of the standard permutation module for $S_m$ over the field with two
elements. This is of dimension $m-1$ when $m$ is odd, and of dimension $m-2$
when $m$ is even. 

The following is Theorem~\ref{T:genbyS3based} below.

\begin{theorem}\label{T:mainsolgen}
Let $G$ be a finite group and $D$ an abelian $2$-group on which $G$ acts
faithfully. Assume that $G$ has no nontrivial normal $2$-subgroups and that $G$
is generated by its solitary offenders on $D$. Then $G$ is a direct product of
symmetric groups of odd degree, and $D/C_D(G)$ is a direct sum of natural
modules.
\end{theorem}

Thus, ultimately, an appeal to the General FF-module Theorem is replaced by an
appeal (in the proof of Theorem~\ref{T:mainsolgen}) to McLaughlin's
classification of irreducible subgroups of $SL_n(2)$ generated by transvections
\cite{McLaughlin1969}. 

The article proceeds as follows. In Section~\ref{S:limits} we state the
required definitions and background results in fusion systems and homological
algebra. Our results on control of fixed points are based on
\cite[Theorem~A1.4]{Glauberman1971}, to which we refer as the \emph{norm
argument}; this is introduced in Section~\ref{S:podd}. We use it there to
handle the case where $p$ is odd. In Section~\ref{S:norm2}, we state a modified
form of the norm argument and show that applies to many cases when $p=2$. In
Section~\ref{S:reduction2}, we reduce to the case in which all minimal
offenders are solitary, and we handle that case in
Section~\ref{S:transvections}.  Appendix \ref{A:norm2} provides a proof of the
modified norm argument of Section~\ref{S:norm2}, while Appendix \ref{A:groups}
provides various miscellaneous group-theoretic results, including a discussion
of conjugacy functors and their well-placed subgroups. 

\subsection*{Notation}
Conjugation maps and morphisms in fusion systems will be written on the right
and in the exponent, while cocycles and cohomology classes for functors will be
written on the left.  Groups of cochains are written multiplicatively. However,
on some occasions we express that a group, or a cocycle or a cohomology class,
is trivial by saying that it is equal to $0$. 

\section*{Acknowledgements}
We would like to extend our thanks to an anonymous referee, whose careful
reading led to numerous improvements in the mathematics and the exposition of
this paper. In particular, we are grateful to the referee for pointing out an
error in a previous version of Lemma~\ref{L:SSpermlocal}.  It is our pleasure
to thank the Department of Mathematics at the University of Chicago for
supporting the visits of the second author, during which this work began. The
second author thanks the department for its hospitality during his stays.
Finally we would like to thank Andrew Chermak and Bob Oliver for their comments
on an earlier draft of this paper. 

\section{Functors on orbit categories}\label{S:limits}

In this section, we recall the definition of a saturated fusion system and some
terminology and constructions from homological algebra that are needed. Our
notation follows \cite{Oliver2013}, and we also recall here some of the
preliminary lemmas on cohomology that we use from that paper. Since we will
have very little explicit need for the theory of fusion or linking systems, we
refer to \cite{AschbacherKessarOliver2011} for the basic properties that may be
tangential to the concerns here.

All groups in this paper are finite. Let $G$ be a group and $g \in G$.  The
conjugation homomorphism induced by $g$ (and its restrictions) is written
$c_g\colon x \mapsto x^g = g^{-1}xg$. Images of morphisms in fusion systems are
also written on the right and in the exponent; for example, the image of $x$
under the morphism $\varphi$ is denoted $x^\varphi$, by analogy with
conjugation.  Given a group $G$ and subgroups $P$, $Q$, we write $\Hom_G(P,Q)$
for the set $\{c_g \mid g \in G \text{ and } P^g \leq Q\}$. The notation
$\Aut_G(P)$ means $\Hom_G(P,P)$; a similar notation is used within categories
employed in this paper. 

\begin{definition}
Let $S$ be a finite $p$-group. A \emph{fusion system} over $S$ is a category
$\F$ with objects the subgroups of $S$ and with morphisms injective group
homomorphisms between subgroups, satisfying
\begin{enumerate}
\item[(1)] $\Hom_S(P,Q) \subseteq \Hom_\F(P,Q)$; and
\item[(2)] each morphism in $\F$ factors as an isomorphism in $\F$ followed
by an inclusion.
\end{enumerate}
A subgroup $Q \leq S$ is \emph{fully $\F$-normalized} if $N_S(Q)$ has largest
order among all such normalizers of subgroups that are $\F$-isomorphic to $Q$.
Similarly, $Q$ is \emph{fully $\F$-centralized} if $C_S(Q)$ has largest order
among all such centralizers of subgroups that are $\F$-isomorphic to $Q$. 
\smallskip
\noindent
The fusion system $\F$ is \emph{saturated} if the following two axioms hold:
\begin{enumerate}
\item[(I)] (Sylow Axiom) each subgroup $Q \leq S$ that is fully $\F$-normalized
is also fully $\F$-centralized, and $\Aut_S(Q)$ is a Sylow $p$-subgroup of
$\Aut_\F(Q)$ in this case; and
\item[(II)] (Extension Axiom) for each fully $\F$-centralized $Q \leq S$, and
each isomorphism $\varphi\colon P \to Q$ in $\F$, if we set 
\[
N_\varphi = \{ s \in N_S(P) \mid \varphi^{-1} c_s \varphi \in \Aut_S(Q) \},
\]
then there exists $\tilde{\varphi} \in \Hom_\F(N_{\varphi}, S)$ such that
$\tilde{\varphi}|_{P} = \varphi$. 
\end{enumerate}
\end{definition}

Fix a saturated fusion system $\F$ over a finite $p$-group $S$. Then two
subgroups of $S$ are $\F$-\emph{conjugate} if they are isomorphic in $\F$.  The
subgroup $P$ is $\F$-\emph{centric} if $C_S(Q) \leq Q$ for every subgroup $Q$
of $S$ that is $\F$-conjugate to $P$.  We write $\F^f$ and $\F^c$ for the
collection of fully $\F$-normalized and the collection of $\F$-centric
subgroups of $S$, respectively. We also use $\F^c$ to denote the full
subcategory of $\F$ with objects the $\F$-centric subgroups. By the Fusion
Theorem of Alperin-Goldschmidt \cite[Theorem~A.10]{BrotoLeviOliver2003}, the
subcategory of $\F$-centrics ``generates'' all $\F$-conjugacy, and so there is
usually no loss of information in restricting attention to this subcategory.

Since morphisms are written on the right, $\Hom_\F(P,Q)$ has a right action by
$\Inn(Q)$ for each pair of subgroups $P,Q \leq S$. The \emph{$\F$-centric orbit
category} $\O(\F^c)$ has as objects the set $\F^c$ and as
morphisms the orbits under this action: 
\[ 
\Mor_{\O(\F^c)}(P,Q) = \Hom_\F(P,Q)/\Inn(Q).  
\] 
The class in the orbit category of a morphism $\phi \in \Hom_\F(P,Q)$ is
denoted by $[\phi]$. 

The \emph{center functor} is the contravariant functor $\Z_\F\colon \O(\F^c)
\to \Ab$ defined by $\Z_\F(P) = Z(P)$ on objects. For a morphism $\phi\colon P
\to Q$ in $\F$, $\Z_\F([\phi])$ is the map from $Z(Q) = C_S(Q) \leq C_S(P^\phi)
= Z(P^\phi)$ to $Z(P)$ induced by $\phi^{-1}$. It is necessary to restrict
objects, for example to the centrics as described here, in order that taking
centers in this way determines a functor.

Filtrations of the collection of subgroups of $S$ yield useful filtrations of
the center functor by subquotient functors.  Denote by $\sS(S)$ the set of
subgroups of $S$, and by $\sS(S)_{\geq Y}$ the subset of those that contain a
fixed subgroup $Y \leq S$.

\begin{definition}
A collection $\R \subseteq \sS(S)$ is an \emph{interval} if $P \in \R$ whenever
$P_1, P_2 \in \R$ and $P_1 \leq P \leq P_2$. An interval $\R$ is
$\F$-\emph{invariant} if $P^\phi \in \R$ whenever $P \in \R$ and $\phi \in
\Hom_\F(P, S)$.
\end{definition}

Given an $\F$-invariant interval $\R \subseteq \F^c$, define the functor
$\Z_\F^\R\colon \O(\F^c) \to \Ab$ by $\Z_\F^\R(P) = Z(P)$ whenever $P \in \R$,
and by $1$ otherwise. Then $\Z^{\R}_\F$ is a subfunctor of $\Z_\F$ when $\R$ is
closed under passing to (centric) subgroups, and a quotient functor of $\Z_\F$
when $\R$ is closed under passing to overgroups in $S$ -- that is, when $S \in
\R$. 

Following \cite{Oliver2013}, we write 
\[
L^k(\F; \R) := \varprojlim{\!}^k_{\O(\F^c)} \Z^\R_\F
\]
for the higher derived limits of these functors, and we think of them as
cohomology groups of the category $\O(\F^c)$ with coefficients in the functor
$\Z^{\R}_\F$. They are cohomology groups of a certain cochain complex
$C^*(\O(\F^c);\Z^\R_\F)$, in which $k$-cochains are maps from sequences of $k$
composable morphisms in the category. A $0$-cochain $u$ is a map sending $P \in
\F^c$ to an element $u(P) \in \Z^\R_\F(P)$, and a $1$-cochain $t$ is a map
sending a morphism $P \xrightarrow{[\phi]} Q$ to an element $t([\phi]) \in
\Z^\R_\F(P)$. We will be working in \S\ref{S:reduction2} with cochains for
$\Z_\F^\Q$ in the case where $\Q$ is closed under passing to overgroups.  With
our notational conventions, the coboundary maps on such $0$- and $1$-cohains in
this special case are as follows:
\begin{align}
du([\phi]) &= 
\begin{cases}
\label{E:0cob}
u(Q)^{\phi^{-1}}u(P)^{-1} & \text{if $P \in \Q$, and}\\
1  & \text{otherwise; and}
\end{cases}\\
dt([\phi][\psi]]) &= 
\begin{cases} 
\label{E:1cob}
t([\psi])^{\phi^{-1}}t([\phi\psi])^{-1}t([\phi]) & \text{if $P \in \Q$, and}\\
1  & \text{otherwise,}
\end{cases}
\end{align}
for any sequence $P \xrightarrow{\phi} Q \xrightarrow{\psi} R$ of composable
morphisms in $\F^c$.  We refer to \cite[\S III.5.1]{AschbacherKessarOliver2011}
for more details on the bar resolution for these functors. 


\begin{definition} 
A \textit{general setup for the prime $p$} is a triple $(\Gamma, S, Y)$ where
$\Gamma$ is a finite group, $S$ is a Sylow $p$-subgroup of $\Gamma$, and $Y$ is
a normal $p$-subgroup of $\Gamma$ such that $C_\Gamma(Y) \leq Y$. A
\textit{reduced setup} is a general setup such that $Y = C_S(Z(Y))$ and
$O_p(\Gamma/C_\Gamma(Z(Y))) = 1$.
\end{definition}

We next state the three preliminary lemmas from \cite{Oliver2013} that are
needed later. 

\begin{lemma}\label{L:olijm}
Let $\F$ be a saturated fusion system over a $p$-group $S$, and let $\Q
\subseteq \F^c$ be an $\F$-invariant interval such that $S \in \Q$. Let $\F_\Q$
be the full subcategory of $\F$ with object set $\Q$, and denote by
$|_\Q$ the restriction of a functor to $\F_\Q$.
\begin{enumerate}
\item[(a)] The inclusion $\F_\Q \to \F^c$ induces an isomorphism of cochain
complexes $C^*(\O(\F^c); \Z_\Q)$ $\xrightarrow{\cong}$ $C^*(\O(\F_\Q);
\Z_\F^\Q|_\Q)$, and in particular an isomorphism 
\[
L^*(\F; \Q) \xrightarrow{\cong} L^*(\F_\Q; \Q) :=
\varprojlim{\!}^*_{\O(\F_\Q)} \Z^\Q_\F|_{\Q}.
\]
\item[(b)] If $(\Gamma, S, Y)$ is a general setup, $\F = \F_S(\Gamma)$, and $\Q
= \sS(S)_{\geq Y}$, then 
\[
L^k(\F;\Q) =
\begin{cases}
C_{Z(Y)}(\Gamma) & \text{if $k = 0$; and}\\
0 & \text{otherwise.}
\end{cases}
\]
\end{enumerate}
\end{lemma}
\begin{proof}
This is Lemma~1.6 of \cite{Oliver2013}, with the additional information in part
(a) shown in its proof.
\end{proof} 

Part (b) of the following lemma gave us the first concrete indication that
questions regarding control of fixed points by $p$-local subgroups would be
relevant to Theorem~\ref{T:main}. It is the starting point for nearly all the
arguments to follow.

\begin{lemma}\label{L:oliexact}
Let $\F$ be a saturated fusion system over a $p$-group $S$. Let $\Q, \R
\subseteq \F^c$ be $\F$-invariant intervals such that 
\begin{enumerate}
\item[(i)] $\Q \cap \R = \varnothing$,
\item[(ii)] $\Q \cup \R$ is an interval,
\item[(iii)] $Q \in \Q$, $R \in \R$ implies $Q \nleq R$. 
\end{enumerate}
Then $\Z_\F^\R$ is a subfunctor of $\Z_\F^{\Q \cup \R}$, $\Z^{\Q \cup
\R}_\F/\Z_\F^\R \cong \Z_\F^\Q$, and there is a long exact sequence
\begin{align*}
0 &\lra L^0(\F;\R) \lra L^0(\F; \Q \cup \R) \lra L^0(\F;\Q) \lra \cdots \\
 &\lra L^{k-1}(\F;\Q) \lra L^k(\F;\R) \lra L^k(\F; \Q \cup \R) \lra L^k(\F;\Q) \lra \cdots.
\end{align*}
In particular, if $(\Gamma, S, Y)$ is a general setup, $D = Z(Y)$, $\F =
\F_S(\Gamma)$ and $\Q \cup \R = \sS(S)_{\geq Y}$, then 
\begin{enumerate}
\item[(a)] $L^{k-1}(\F;\Q) \cong L^k(\F;\R)$ for each $k \geq 2$, and 
\item[(b)] there is a short exact sequence
\[
1 \lra C_D(\Gamma) \lra C_D(\Gamma^*) \lra L^1(\F;\R) \lra 1,
\]
where $\Gamma^*$ is the set of $g \in \Gamma$ such that there exists $Q \in \Q$
with $Q^g \in \Q$.
\end{enumerate}
\end{lemma}
\begin{proof}
This is nearly Lemma~1.7 of \cite{Oliver2013}. Our part (a) is stated in the
situation of a general setup, and it follows from the long exact sequence and
Lemma~\ref{L:olijm}(b).
\end{proof}

\begin{lemma}\label{L:olirestinj}
Let $(\Gamma, S, Y)$ be a general setup, $\Gamma_0$ a normal subgroup of
$\Gamma$ containing $Y$, and $S_0 = S \cap \Gamma_0$. Set $\F = \F_S(\Gamma)$
and $\F_0 = \F_{S_0}(\Gamma_0)$. Let $\Q \subseteq \sS(S)_{\geq Y}$ be an
$\F$-invariant interval such that $S \in \Q$, and such that $\Gamma_0 \cap Q
\in \Q$ whenever $Q \in \Q$. Set $\Q_0 = \{Q \in \Q \mid Q \leq \Gamma_0\}$.
Then restriction induces an injection
\[
L^1(\F;\Q) \lra L^1(\F_0; \Q_0).
\]
\end{lemma}
\begin{proof}
This is Lemma~1.13 of \cite{Oliver2013}.
\end{proof}

\section{The norm argument and the odd case}\label{S:podd}

For a finite group $G$ with action on an abelian group $V$ (written
multiplicatively), and a subgroup $H$ of $G$, the \textit{norm map}
$\N_H^G\colon C_V(H) \to C_V(G)$ is defined by
\[
\N_H^G(v) = \prod_{g \in [G/H]} v^g
\]
for each $v \in C_V(H)$, where $v^g$ denotes the image of $v$ under $g$, and
where $[G/H]$ is a set of right coset representatives for $H$ in $G$. We say
that $\N_H^G = 1$ \textit{on $V$} if this map is constant, onto the
identity element of $V$. Since $\N_H^G = \N_K^G\N_H^K$ whenever $H \leq K
\leq G$, one sees that $\N_H^G = 1$ on $V$ whenever either of $\N_K^G$ or
$\N_H^K$ is $1$ on $V$. 

In this section, we give some sufficient conditions for determining that the
norm map $\N_H^G$ is constant for suitable $p$-local subgroups $H$ of $G$, and
then apply these results in Proposition~\ref{P:3.3odd} to obtain a proof of
Theorem~\ref{T:main} for odd primes. 

A subgroup $A$ of $G$ acts \emph{quadratically} on $V$ if $[V,A,A] = 1$ but
$[V,A] \neq 1$. In particular, when $V$ is elementary abelian, each element of
such a subgroup has quadratic or linear minimum polynomial in its action on
$V$.

\begin{lemma}\label{L:quadnorm}
Suppose that $A$ is a $p$-group acting on an elementary abelian $p$-group $V$. 
\begin{enumerate}
\item[(a)] if $p$ is odd and $A$ acts quadratically on $V$, then $\N_{A_0}^A =
1$ on $V$ for every proper subgroup $A_0$ of $A$.  
\item[(b)] if $p = 2$, then $\N_{A_0}^A = 1$ on $V$ for every subgroup $A_0$ of
$A$ satisfying one of the following conditions:
\begin{enumerate}
\item[(i)] $|A:A_0| \geq 2$ and $C_V(A_0) = C_V(A)$, or
\item[(ii)] $|A: A_0| \geq 4$ and $A$ acts quadratically on $V$.
\end{enumerate}
\end{enumerate}
\end{lemma}
\begin{proof}
By abuse of notation, we use the same symbol to denote an element of $A$ and
the endomorphism it induces on $V$. Suppose first that $p$ is odd.  Let $A_0
\leq A_1 \leq A$ with $A_1$ of index $p$ in $A$, and $a \in A-A_1$.  Then
$(1-a)^2 = 0$ in $\End(V)$ by assumption and $pa = 0$ in $\End(V)$ since $V$ is
elementary abelian. Hence
\[
\N_{A_1}^A(v) = v^{1+a+\cdots+a^{p-1}} = v^{(1-a)^{p-1}} = 1 
\]
for all $v \in C_V(A_1)$, since $p-1 \geq 2$. Thus $\N_{A_0}^A(v) =
\N_{A_1}^{A}(\N_{A_0}^{A_1}(v)) = 1$ for all $v \in C_V(A_0)$ as desired.

Let $p=2$. Under assumption (i), choose coset representatives $\{1,a\}$ for a
maximal subgroup $A_1$ of $A$ containing $A_0$, and then $\N_{A_1}^A(v) = vv^a
= v^2 = 1$ for all $v \in C_V(A_1) = C_V(A)$, since $V$ is elementary abelian.
Under assumption (ii), fix a subgroup $A_1$ containing $A_0$ with index $4$ in
$A$ and choose coset representatives for $A_1$ in $A$ as $\{1,a,b,ab\}$ for
suitable $a, b \in A-A_1$. Then $1 + a + b + ab = (1-a)(1-b) = 0$ in $\End(V)$
since the action is quadratic, and so $\N_{A_1}^A(v) = v^{1+a+b+ab} = 1$ for
all $v \in C_V(A_1)$. As in the odd case, this completes the proof.
\end{proof}

\begin{theorem}\label{T:glawc}
Suppose $G$ is a finite group, $S \in \Syl_p(G)$, and $D$ is a $p$-group on
which $G$ acts.  Let $\A$ be a nonempty set of subgroups of $S$, and set $J =
\gen{\A}$.  Assume that $J$ is weakly closed in $S$ with respect to $G$ and
that
\begin{align}\label{E:normcond}
\parbox[t]{0.85\linewidth}{whenever $A \in \A$, $g \in G$, $A \nleq S^g$, and
$V$ is a composition factor of $D$ under $G$, then $\N_{A \cap S^g}^{A} = 1$ on
$V$.}
\end{align}
Then $C_D(N_G(J)) = C_D(G)$.
\end{theorem}
\begin{proof}
This is Theorem~A1.4 of \cite{Glauberman1971}. 
\end{proof}

We refer to Theorem~\ref{T:glawc}, and also to Theorem~\ref{T:glawc2} below, as
the \emph{norm argument} for short. In the past, it has usually been applied
with $p$ odd and in the presence of quadratic elements in $G$ on $D$ (cf.
Lemma~\ref{L:quadnorm}(a)).  

\begin{definition}\label{D:gensetupnotation}
For a general setup $(\Gamma, S, Y)$, set $D = Z(Y)$ and $G =
\Gamma/C_\Gamma(D)$. Let $\A$ be a set of abelian subgroups of $G$ that is
invariant under $G$-conjugation.  For any subgroup $H$ of $G$, set $\A \cap H =
\{A \in \A \,|\, A \leq H\}$ and $J_\A(H) = \gen{\A \cap H}$.  For a subgroup
$H$ of $\Gamma$, we let $J_\A(H, D)$ be the preimage in $H$ of
$J_\A(HC_\Gamma(D)/C_\Gamma(D))$. Often $J_\A(H,D)$ will be abbreviated to
$J_\A(H)$ when $D$ is understood. 
\end{definition}

The collection $\A$ of Definition~\ref{D:gensetupnotation} will generally be
some subset of the collection of nontrivial best offenders in $G$ on $D$, as
defined below. 

\begin{definition}\label{D:offender}
Let $G$ be a finite group and let $D$ be an abelian $p$-group on which $G$ acts
faithfully. An abelian $p$-subgroup $A \leq G$ is an \emph{offender} on $D$ if 
\[
|A||C_D(A)| \geq |D|
\]
and a \emph{best offender} if
\[
|A||C_D(A)| \geq |B||C_D(B)|
\]
for every subgroup $B \leq A$. An offender $A$ is \emph{nontrivial} if $A \neq
1$.  We write $\A_D(G)$ for the collection of nontrivial best offenders in $G$
on $D$. 
\end{definition}

A best offender is, in particular, an offender, as can be seen by taking $B=1$
in the above definition. Conversely, each nontrivial offender $A$ contains a
nontrivial best offender, which can be obtained as a subgroup $B$ such that the
quantity $|B||C_D(B)|$ is maximal among all nontrivial subgroups of $A$.  In
turn, by the Timmesfeld Replacement Theorem \cite{Timmesfeld1982}, each
nontrivial best offender contains a nontrivial \emph{quadratic best offender},
namely, a best offender that acts quadratically on $D$. We include a short proof
of this, using the Thompson Replacement Theorem, in the form which is needed
here.

\begin{lemma}\label{L:replacement}
Suppose $A$ is a nontrivial offender on $D$.  Let $B$ be a nontrivial subgroup
of $A$ that is minimal under inclusion subject to $|B||C_D(B)| \geq
|A||C_D(A)|$.  Then $B$ is a quadratic best offender on $D$. 
\end{lemma}
\begin{proof}
It follows from the choice of $B$ that $B$ is a best offender, so we need only
show that it acts quadratically on $D$. We work in the semidirect product $DB$,
where we set $C = BC_D(B)$. 

We first show that $C$ is an abelian subgroup of $DB$ of maximum possible
order. Suppose that $C_1$ is an abelian subgroup of $DB$ such that $|C_1| \geq
|C|$, and let $B_1$ be the image of $C_1$ under the projection of $DB$ onto
$B$. Then $C_1 \cap D \leq C_D(C_1) = C_D(B_1)$, and so
\[
|C| \leq |C_1| = |C_1/(C_1 \cap D)||C_1 \cap D| = |B_1||C_1 \cap D| \leq
|B_1||C_D(B_1)| \leq |B||C_D(B)| = |C|
\]
with the last inequality since $B$ is a best offender on $D$. Therefore
equality holds everywhere, which yields
\begin{eqnarray}\label{E:bestoff2}
C_1 \cap D = C_D(B_1) \quad \text{ and } \quad |C_1| = |C| = |B_1C_D(B_1)| = |B||C_D(B)|. 
\end{eqnarray}
This shows that $C$ is an abelian subgroup of maximal order in $DB$. 

Note that if $D$ normalizes $C$, then $[D,C] \leq C$ and so $[D,B,B] =
[D,C,C] \leq [C,C] = 1$ since $C$ is abelian. Hence $B$ acts quadratically on
$D$ in this case.

Suppose that $D$ does not normalize $C$. Then by Thompson's Replacement Theorem
\cite[Theorem~8.2.5]{Gorenstein1980}, there exists an abelian subgroup $C_1$ of
$DB$ such that $|C_1| = |C|$, $C_1 \cap D > C \cap D$, and $C_1$ normalizes
$C$. Take $B_1$ as above. Then by \eqref{E:bestoff2}, 
\[
|B_1||C_D(B_1)| = |B||C_D(B)|
\]
and 
\[
|B_1| = |B||C_D(B)|/|C_D(B_1)| = |B||C \cap D|/|C_1 \cap D| < |B|.
\]
By the minimal choice of $B$, we have that $B_1 = 1$ and therefore that $C_1 =
C_D(B_1) = D$. This shows that $D$ normalizes $C$, a contradiction which
completes the proof of the lemma.
\end{proof}

\begin{lemma}\label{L:normquot}
Let $(\Gamma, S, Y)$ be a general setup for the prime $p$, set $D = Z(Y)$, and
use bar notation for images modulo $C_\Gamma(D)$.  If $Q$ is a subgroup of $S$
containing $C_S(D)$, then $\ol{N_\Gamma(Q)} = N_{\bar{\Gamma}}(\bar{Q})$. 
\end{lemma}
\begin{proof}
Let $N$ be the preimage of $N_{\bar{\Gamma}}(\bar{Q})$. Then
$C_\Gamma(D)N_\Gamma(Q) \leq N$ and we must show that $N \leq
C_\Gamma(D)N_\Gamma(Q)$. Now $QC_\Gamma(D)$ is normal in $N$, and $QC_\Gamma(D)
\cap S = QC_S(D) = Q$ is a Sylow $p$-subgroup of $QC_\Gamma(D)$. By the
Frattini argument,
\[
N = (QC_\Gamma(D))N_N(Q) \leq QC_\Gamma(D)N_\Gamma(Q) = C_\Gamma(D)N_\Gamma(Q)
\]
as desired. 
\end{proof}

The following proposition is a generalization of
\cite[Proposition~3.2]{Oliver2013} for odd primes.

\begin{proposition}\label{P:3.2}
Let $(\Gamma, S, Y)$ be a general setup for the prime $p$. Set $D = Z(Y)$, $G =
\Gamma/C_\Gamma(D)$, and $\F = \F_S(\Gamma)$. Let $\A$ be a $G$-invariant
collection of $p$-subgroups of $G$, each of which acts nontrivially and
quadratically on $D$. Let $\R \subseteq \sS(S)_{\geq Y}$ be an $\F$-invariant
interval such that $Y \in \R$ and $J_\A(S) \notin \R$.  If $p$ is odd, then
$L^1(\F;\R) = 0$.
\end{proposition}
\begin{proof}
Set $\Q = \sS(S)_{\geq Y} - \R$. Let $\Gamma^*$ be the subset of $\Gamma$
consisting of those $g \in \Gamma$ for which there is $Q \in \Q$ with $Q^g \in
\Q$. Then $\Q$ and $\R$ are $\F$-invariant intervals that satisfy the
hypotheses of Lemma~\ref{L:oliexact}, so
\begin{eqnarray}\label{E:incliso}
\text{it suffices to show that $C_{D}(\Gamma) = C_{D}(\Gamma^*)$}
\end{eqnarray}
by part (b) of that lemma. 

As each element of $\A$ acts quadratically on $D$, Lemma~\ref{L:quadnorm}(a)
shows that \eqref{E:normcond} is satisfied.  Hence, $C_D(\Gamma) =
C_D(N_\Gamma(J_\A(S)))$ by Theorem~\ref{T:glawc} and Lemma~\ref{L:normquot}
(where the latter applies by Definition~\ref{D:gensetupnotation}).  Since
$J_\A(S) \in \Q$ by assumption, we have $N_\Gamma(J_\A(S)) \leq \Gamma^*$.
Hence
\[
C_D(\Gamma) = C_D(N_\Gamma(J_\A(S))) \geq C_D(\Gamma^*) \geq C_D(\Gamma)
\]
and \eqref{E:incliso} complete the proof.
\end{proof}

At odd primes, Theorem~\ref{T:main} now follows from
Proposition~\ref{P:3.2}, \cite[Proposition~3.3]{Oliver2013}, and the proof of
\cite[Theorem~3.4]{Oliver2013}.  However, when compared with
\cite[Proposition~3.2]{Oliver2013}, the increased generality of
Proposition~\ref{P:3.2} renders unnecessary some steps of the proof of
Proposition~3.3 of \cite{Oliver2013}. The following is Oliver's Proposition~3.3
with a simplified proof.

\begin{proposition}\label{P:3.3odd}
Let $(\Gamma, S, Y)$ be a general setup for the prime $p$. Set $\F =
\F_S(\Gamma)$, $D = Z(Y)$, and $G = \Gamma/C_\Gamma(D)$.  Let $\R \subseteq
\sS(S)_{\geq Y}$ be an $\F$-invariant interval such that for all $Q \in
\sS(S)_{\geq Y}$, $Q \in \R$ if and only if $J_{\A_D(G)}(Q) \in \R$. If $p$ is
odd, then $L^k(\F;\R) = 0$ for all $k \geq 1$. 
\end{proposition}
\begin{proof}
Let $(\Gamma,S,Y,\R,k)$ be a counterexample for which the four-tuple
$(k,|\Gamma|,|\Gamma/Y|, |\R|)$ is minimal in the lexicographic ordering.
Steps 1 and 2 in the proof of \cite[Proposition~3.3]{Oliver2013} show that $\R
= \{P \leq S \mid J_{\A_D(G)}(P) = Y\}$ and $k = 1$ (since $p$ is odd). 

Let $\A$ be the set of nontrivial best offenders in $G$ on $D$ that are minimal
under inclusion. Each best offender contains a member of $\A$ as a subgroup,
and so $J_{\A_D(G)}(P) = Y$ if and only if $J_\A(P) = Y$.  By
Lemma~\ref{L:replacement}, each member of $\A$ acts quadratically on $D$. 

Clearly, $Y \in \R$. If $S \in \R$, then $\R = \sS(S)_{\geq Y}$ since $\R$ is
an interval, and $L^k(\F;\R) = 0$ for all $k \geq 1$ by Lemma~\ref{L:olijm}(b).
Hence $S \notin \R$ and so $J_\A(S) \notin \R$. Now Proposition~\ref{P:3.2}
shows that $(\Gamma, S, Y, \R, 1)$ is not a counterexample.  
\end{proof}

\section{Norm arguments for $p=2$}\label{S:norm2}

In this section, we define the notion of a solitary offender and prove the
lemmas necessary to obtain Theorem~\ref{T:solobs}. These results are used in
\S\ref{S:reduction2} to give a proof of \cite[Proposition~3.2]{Oliver2013}
except in the case where every minimal offender under inclusion is solitary.

The following analogue of Theorem~\ref{T:glawc} is easier to apply in
applications for $p=2$. The proof is very similar to that of
Theorem~\ref{T:glawc} and is given in Appendix~\ref{A:norm2} as
Theorem~\ref{T:norm2-appendix}.

\begin{theorem}\label{T:glawc2}
Suppose $G$ is a finite group, $S$ is a Sylow $2$-subgroup of $G$, and $D$ is
an abelian $2$-group on which $G$ acts. Let $\A$ be a nonempty set of subgroups
of $S$, and set $J = \gen{\A}$. Let $H$ be a subgroup of $G$ containing
$N_G(J)$. Set $V = \Omega_1(D)$. Assume that $J$ is weakly closed in $S$ with
respect to $G$, and
\begin{align}\label{E:normcond2}
\parbox[t]{0.85\linewidth}{whenever $A \in \A$, $g \in G$, and $A \nleq H^g$,
then $\N_{A \cap H^g}^{A} = 1$ on $V$,}
\end{align}
or more generally,
\begin{align}\label{E:gennormcond2}
\parbox[t]{0.85\linewidth}{whenever $g \in G$ and $J \nleq H^g$,
then $\N_{J \cap H^g}^{J} = 1$ on $V$.}
\end{align}
Then $C_D(H) = C_D(G)$. 
\end{theorem}

Throughout the remainder of this section, we fix a finite group $G$, a Sylow
$2$-subgroup $S$ of $G$, an abelian $2$-group $D$ on which $G$ acts faithfully,
and we set $V = \Omega_1(D)$. 

\begin{definition}
Set
\begin{align*}
\hat{\A}_D(G) = \{A \in \A_D(G) \mid |A||C_D(A)| > |D|\},
\end{align*}
the members of which are sometimes called \emph{over-offenders}.

Denote by $\A_D(G)^\circ$ the set of those members of $\A_D(G)$ that are
minimal under inclusion, and denote by $\hat{\A}_D(G)^\circ$ those members of
$\hat{\A}_D(G)$ that are minimal under inclusion.

For a positive integer $k$, write
\begin{align*}
\A_D(G)_2 &= \{A \in \A_D(G) \mid |A| = 2\}; \text{ and} \\
\A_D(G)^\circ_{\geq 4} &= \{A \in \A_D(G)^\circ \mid |A| \geq 4\}.
\end{align*}
\end{definition}

It may help to reiterate that a member of $\hat{\A}_D(G)^\circ$, while minimal
under inclusion in the collection $\hat{\A}_D(G)$, may not be minimal under
inclusion in $\A_D(G)$. By Lemma~\ref{L:replacement}, each member of
$\A_D(G)^\circ$ acts quadratically on $D$.

\begin{remark}\label{R:Ahat4}
If $A \in \A_D(G)_2$, then $|D/C_D(A)| = 2$.  In particular, every member of
$\hat{\A}_D(G)$ is of order at least $4$.
\end{remark}

\begin{lemma}\label{L:controlbig}
Let $\A = \A_D(G)^\circ_{\geq 4} \cup \hat{\A}_D(G)^\circ$. Assume that $\A$ is
not empty and $H$ is a subgroup of $G$ containing $N_G(J_\A(S))$. Then $\A \cap
S$ satisfies \eqref{E:normcond2}.  
\end{lemma}
\begin{proof}
Fix $A \in \A \cap S$ and $g \in G$ with $A \nleq H^g$, and let $A_0$ be a
subgroup of $A$ of index $2$ that contains $A \cap H^g$. Suppose $\N_{A \cap
H^g}^A$ is not $1$ on $V$. Since $\N_{A \cap H^g}^A = \N_{A_0}^A\N_{A \cap
H^g}^{A_0}$, we have $C_V(A) < C_V(A_0)$ by Lemma~\ref{L:quadnorm}. It follows
that
\begin{eqnarray}\label{E:ineq}
|A_0||C_D(A_0)| \geq \frac{1}{2}|A|\cdot 2|C_D(A)| = |A||C_D(A)|.
\end{eqnarray}
If $A \in \A_D(G)^\circ_{\geq 4}$, then $A$ is a best offender minimal under
inclusion subject to $A \neq 1$, so we have $A_0 = 1$ and $|A|=2$. This
contradicts $|A| \geq 4$.  Hence $A \in \hat{\A}_D(G)^\circ$ and in particular
$|A| \geq 4$. But then $A_0 \in \hat{\A}_D(G)$ by \eqref{E:ineq}, which
contradicts the minimality of $A$. 
\end{proof} 

\begin{lemma}\label{L:nooveroffenders}
Let $\A = \A_D(G)_2$. Assume that $\hat{\A}_D(G) = \varnothing$ and that $A, B
\in \A$. Then 
\begin{enumerate}
\item[(a)] if $[A,B] = 1$ and $A \neq B$, then $C_D(A) \neq C_D(B)$ and $AB$ is
quadratic on $D$;
\item[(b)] if $\gen{A, B}$ is a $2$-group, then $[A,B] = 1$; and
\item[(c)] $J_\A(S)$ is elementary abelian.
\end{enumerate}
\end{lemma}
\begin{proof}
Since $A$ and $B$ lie in $\A_D(G)_2$,
\begin{eqnarray}\label{E:A,Bquadratic}
[D,A,A] = 1 = [D,B,B] \text{ and } |D/C_D(A)| = 2.
\end{eqnarray}
Suppose that $[A,B] = 1$ and $A \neq B$, but that $C_D(A) = C_D(B)$. Then $AB$
is of order $4$ and so $C_D(C) = C_D(A)$ for every nontrivial subgroup $C$ of
$AB$. Since
\[
|AB||C_D(AB)| = |A||B||C_D(A)| = 2|D| > |D|,
\]
we have $AB \in \hat{\A}_D(G)$ contrary to assumption.  Thus, the first
statement of (a) holds. Since $A \in \A$, $[D, A]$ is $B$-invariant and of
order $2$. Thus $[D,A,B] = 1$. By symmetry $[D,B,A] = 1$.  By a commutator
identity and \eqref{E:A,Bquadratic}, $[D,AB,AB] = 1$, which establishes the
second statement of (a).

Suppose $\gen{A,B}$ is a $2$-group and $[A,B] \neq 1$. Let $X = \gen{A,B}$ and
$C = [A,B]$. Then $[D,C] \neq 1$ because $G$ is faithful on $D$. If $[D,A,B] =
[D,B,A] = 1$, then $1 = [A,B,D] = [C,D]$ by the Three Subgroups Lemma, a
contradiction. By the symmetry between $A$ and $B$,
\begin{eqnarray}
\label{E:dabdba}
[D,A,B] \neq 1.
\end{eqnarray}

By \eqref{E:A,Bquadratic} and \eqref{E:dabdba}, 
\[
[C_D(A),B] \geq [D,A,B] > 1.
\]
Hence, $|D/C_D(X)| = |D/(C_D(A) \cap C_D(B)| = 4$. Since $X$ is a $2$-group
that acts on $D/C_D(X)$, there is an $X$-invariant subgroup $D_1$ such that
$C_D(X) < D_1 < D$. As $D/D_1$ has order $2$, $[D,A] \leq D_1$. 

If $C_D(A) \neq D_1$, then $[D,A] \leq C_D(A) \cap D_1 = C_D(X)$, and $[D,A,B]
= 1$, contrary to \eqref{E:dabdba}. So $C_D(A) = D_1$. As $D_1/C_D(X)$ is of
order $2$, $[D_1,X] \leq C_D(X)$ and $[D_1,X,X] \leq [C_D(X),X] = 1$. Then
$[X,D_1,X] = [D_1,X,X] = 1$. Invoking the Three Subgroups Lemma again, we obtain
\begin{eqnarray}
\label{E:XXD1}
[X,X,D_1] = 1.
\end{eqnarray}

Take $C^*$ of order $2$ inside $C$. Then $[C^*,D_1] \leq [C,D_1] = 1$ by
\eqref{E:XXD1} and $C_D(A) = D_1 \leq C_D(C^*)$. Hence $C_D(A) = C_D(C^*)$
because $G$ is faithful on $D$. Therefore $C^* \in \A$. By (a), $[A,C^*] \neq 1$.
So $[A,C^*,D] \neq 1$.

However, $[D,A,C^*] \leq [C_D(A), C^*] = 1$ and similarly $[D,C^*,A] = 1$.  So
the Three Subgroups Lemma yields $[A,C^*,D] = 1$, a contradiction. This proves
(b). Part (c) then follows immediately from (b).
%
\end{proof}

\begin{lemma}
\label{L:nota2group}
Let $A, B \in \A_D(G)_2$ and set $L = \gen{A,B}$. If $L$ is not a $2$-group,
then $L \cong S_3$, $[D,L]$ is elementary abelian of order $4$, $[D,L,L] =
[D,L]$, and $D = [D,L] \times C_D(L)$. 
\end{lemma}
\begin{proof}
Since $A$ and $B$ are of order $2$, $L$ is a dihedral group.  Let $K$ be the
largest odd order subgroup of $L$. Then
\[
|D/C_D(K)| \leq |D/C_D(L)| \leq |D/C_D(A)||D/C_D(B)| = 4.
\]
Since $K$ has odd order and $D$ is a $2$-group, $K$ acts faithfully on
$D/C_D(K)$ and $D = [D,K] \times C_D(K)$. So $D/C_D(K)$ is elementary abelian
of order $4$, $|K| = 3$, and $C_D(K) = C_D(L)$. Hence $L$ acts faithfully on
$[D,K]$. We conclude that $L \cong S_3$, and that $[D,K] = [D,L] = [D,L,L]$ as
$[D,L]$ has the standard action of $L$.
\end{proof}

\begin{definition}\label{D:S3based}
Let $\A = \A_D(G)_2$, $J = J_\A(S)$, and $T \in \A \cap S$.  We say that $T$ is
\textit{solitary in} $G$ \textit{relative to} $S$ if there is a subgroup $L$ of
$G$ containing $T$ such that
\begin{enumerate}
\item[(S1)] $L \cong S_3$;
\item[(S2)] $J = T \times C_J(L)$ and $C_J(L) = \gen{(\A \cap S)-\{T\}}$; and
\item[(S3)] $D = [D,L] \times C_D(L)$ and $[D, L, C_J(L)] = 1$. 
\end{enumerate}
For a subgroup $S_0$ of $S$, we say that $T \leq S_0$ is \emph{semisolitary
relative to} $S_0$, if there are subgroups $W$ and $X$ of $D$ that are
normalized by $\gen{\A \cap S_0}$, such that 
\begin{enumerate}
\item[(SS2)] $\gen{\A \cap S_0} = T \times \gen{(\A \cap S_0) - \{T\}}$; and
\item[(SS3)] $W$ is elementary abelian of order $4$, $D = W \times X$, $T$
centralizes $X$, and $\gen{\A \cap S_0 -\{T\}}$ centralizes $W$.
\end{enumerate}

Denote by $\T_D(G)$ the collection of subgroups of $G$ that are solitary
relative to some Sylow $2$-subgroup of $G$.  Likewise, a member of $\A_D(G)_2$
is said to be semisolitary if it is semisolitary relative to some Sylow
$2$-subgroup of $G$.  
\end{definition}

\begin{remark}\label{R:S3basedimpliesS3similar}
Given a subgroup $T$ which is solitary in $G$ relative to $S$, and given $L$ as
in Definition~\ref{D:S3based}, we may take $W = [D,L]$ and $X = C_D(L)$ and see
from Lemma~\ref{L:nota2group} that $T$ is semisolitary relative to $S$. 
\end{remark}

A solitary offender is of order $2$ (by definition), and thus is generated by a
transvection when $D$ is elementary abelian. If $G = S_3$ and $D = C_2 \times
C_2$, then one may take $L = G$ to see that each subgroup of order $2$ in $G$
is solitary. More generally, if $G$ a symmetric group of odd degree and $D$ is
a natural module for $G$, then each subgroup generated by a transposition is
solitary.  Indeed, given a Sylow $2$-subgroup $S$ of $G$ containing the
transposition, $L$ may be taken in this case to move only three points, namely
the two points moved by the transposition and the unique point fixed by $S$.
On the other hand, (S2) implies that $SL_n(2)$ ($n \geq 3$) and even degree
symmetric groups, for example, have no solitary offenders on their respective
natural modules, despite being generated by transvections. We refer the reader
to Lemma~\ref{L:soltrans} for more details. 

There is no ``(SS1)'' in Definition~\ref{D:S3based} because we view (SS2) and
(SS3) as weakenings of (S2) and (S3).  The reader should view the introduction
of semisolitary offenders as auxiliary. They are used in relative situations
when the connection between solitary offenders in $G$ and those in a subgroup
$H$ is difficult to ascertain. The following elementary lemma addresses a
similar uncertainty.

\begin{lemma}\label{L:TcapS}
Assume that $\hat{\A}_D(G) = \varnothing$.  Then $\T_D(G) \cap S$ is the set of
subgroups that are solitary in $G$ relative to $S$.
\end{lemma}
\begin{proof}
One containment follows from the definition.  For the other containment, set
$\A = \A_D(G)_2$ and assume $T \in \T_D(G) \cap S$.  Suppose $T$ is solitary
relative to the Sylow $2$-subgroup $S_1$ of $G$, and let $g \in G$ with $S_1^g
= S$. Set $J_1 = \gen{\A \cap S_1}$ and $J = \gen{\A \cap S}$. Then $J_1$ is
elementary abelian by Lemma~\ref{L:nooveroffenders}(c), and $T^g \leq
J_1^g = J$.  Fix a subgroup $L_1 \cong S_3$ containing $T$ so that (S1)-(S3)
holds with $S_1$ and $J_1$ in the roles of $S$ and $J$, respectively.  Since
$J$ is abelian and weakly closed in $S$ with respect to $G$, there is $h \in
N_G(J)$ with $T^{gh} = T$ (by Lemma~\ref{L:burnside}).  Setting $L = L_1^{gh}$,
one establishes the validity of (S1)-(S3) in Definition~\ref{D:S3based} for
$L$, $S$, and $J$ from their validity for $L_1$, $S_1$, and $J_1$.   
\end{proof}

The following lemma will be applied later with $J^*$ equal to the subgroup
generated by members of $\A_D(G)_2 \cap S$ that are not semisolitary relative
to $S$.

\begin{lemma}\label{L:nonS3basednorm}
Set $\A = \A_D(G)_2$, $\T = \T_D(G)$, and $\B = \A - \T$. Assume  
\begin{enumerate}
\item[(a)] $\hat{\A}_D(G) = \varnothing$; and
\item[(b)] $\B \neq \varnothing$. 
\end{enumerate}
Then for each subgroup $J^* \leq J_\B(S)$ that is weakly closed in $S$ with
respect to $G$, $\A \cap S$ satisfies \eqref{E:gennormcond2} with $H =
N_G(J^*)$.
\end{lemma}
\begin{proof}
Set $J = J_\A(S)$ for short, and let $J^*$ be a subgroup of $J_\B(S)$ that is
weakly closed in $S$ with respect to $G$. Then $J^* \leq J_\B(S) \leq J$ since
$\B \subseteq \A$, and all of these are weakly closed in $S$ with respect to
$G$, since $\A$ and $\B$ are $G$-invariant.
From (a) and Lemma~\ref{L:nooveroffenders}(c) 
\begin{eqnarray}\label{E:Jabelian}
J \text{ is elementary abelian.}
\end{eqnarray}
Since $J^*$ is weakly closed, 
\begin{eqnarray}
\label{E:HcontainsNJ}
S \leq N_G(J) \leq N_G(J^*) = H.
\end{eqnarray}

Assume \eqref{E:gennormcond2} is false. That is, let $g \in G$ with $J \nleq
H^g$, set $I = J \cap H^g$, and suppose $\N_I^J \neq 1$ on $V = \Omega_1(D)$. 

Suppose first that $|J:I| \geq 4$. Let $J_0$ be a subgroup of $J$ with $I \leq
J_0$ and $|J:J_0| = 4$. By \eqref{E:Jabelian}, and since $J$ is generated by
members of $\A_D(G)_2$, we may write $J = J_0 \times A \times B$ with $A, B \in
\A$. Then $\N_I^J = \N_{J_0}^J\N_{I}^{J_0}$ and so $\N_{J_0}^J \neq 1$ on $V$.
However, $AB$ acts quadratically on $C_V(J_0)/C_V(J)$ by
Lemma~\ref{L:nooveroffenders}(a). Thus, $\N_{J_0}^J = 1$ on $V$ by
Lemma~\ref{L:quadnorm}(b)(ii), a contradiction. We conclude that
\begin{eqnarray}\label{E:|J:I|=2} |J\colon I| = 2.
\end{eqnarray}

Fix $T \in \A$ with $J = IT = I \times T$. Let $A \in \A \cap S$ and suppose
that $A \neq T$. Let $B = AT \cap I$. Then $B$ is of order $2$ and $|D/C_D(B)|
\leq |D/C_D(A)||D/C_D(T)| = 4$. If $|D/C_D(B)| = 4$, then 
\[
C_D(I) \leq C_D(B) = C_D(A) \cap C_D(T) \leq C_D(T).
\]
Consequently, $J = IT$ centralizes $C_D(I)$, and $C_D(I) = C_D(J)$. Hence, also
$C_V(I) = C_V(J)$.  So $\N_I^J = 1$ on $V$ in this case by
Lemma~\ref{L:quadnorm}(b)(i), a contradiction.  Thus, $|D/C_D(B)| = 2$. That
is, $B \in \A$. This shows
\begin{eqnarray}\label{E:AT=BT}
\text{\qquad if $A \in \A \cap S$, then $A = T$  or there exists $B \in \A \cap I$ such that
$AT = BT$.}
\end{eqnarray}
In particular, $J = \gen{\A \cap S} = \gen{\A \cap I}T$, and so
\begin{eqnarray}\label{E:Igen}
I = \gen{\A \cap I}.
\end{eqnarray}

Recall that $I = J \cap H^g$, so that $I^{g^{-1}} \leq H$. Let $h \in H$ such
that $I^{g^{-1}h} \leq S$.  Then
we have 
\[
I^{g^{-1}h} = \gen{\A \cap I}^{g^{-1}h} \leq \gen{\A \cap S} = J
\]
by \eqref{E:Igen}. As $H^{h^{-1}g} = H^g$, we may replace $g$ by $h^{-1}g$ for
convenience so that
\begin{eqnarray}\label{E:IinJcapJ^g}
I = J \cap H^g \leq J \cap J^g.
\end{eqnarray}

We now show that $T$ is solitary in $G$ relative to $S$.  Since $J^g \neq J$,
we may choose $U \in (\A \cap S)^g$ such that $U \nleq J$. Then $U \leq J^g$.
Let $L = \gen{T,U}$.  Since $T \leq J$ and $J$ is abelian, 
\[
[I,L] = 1
\]
by \eqref{E:IinJcapJ^g}. 

We first show that $T$ and $U$ do not commute. Suppose on the contrary that
$[T,U] = 1$. Then $IL$ is a $2$-group generated by members of $\A$, and hence
is conjugate to a subgroup of $J$. Then since $J = IT \leq IL$, we see that $J
= IL$. Thus $U \leq L \leq J$, contrary to the choice of $U$. 

Thus,
\begin{eqnarray}\label{E:s3based1} 
L \cong S_3, \,\, |[D,L]| = 4, \text{ and }
D = [D,L] \times C_D(L)
\end{eqnarray}
by (a) and Lemma~\ref{L:nota2group}, with the factors $[D,L]$ and
$C_D(L)$ invariant under $J$. Further, $C_J(L) = I$ by the structure of $L$,
and so $J = T \times I = T \times C_J(L)$ by choice of $T$. 

Since $[D,T]$ is of order $2$ and $I$-invariant, $[D,T,I] = 1$. Since $I^x = I$
and $[D,T^x] = [D,T]^x \leq C_D(I^x) = C_D(I)$ for every $x \in L$,
\begin{eqnarray}\label{E:DLI=1}
[D,L] \leq C_D(I) = C_D(C_J(L)).
\end{eqnarray}

We have shown that (S1), (S3), and half of (S2) hold in
Definition~\ref{D:S3based}; it remains to prove that for each $A \in \A \cap
S$, either $A = T$ or $A \leq I$. Fix $A \in \A \cap S$ and suppose that $A
\neq T$ and $A$ is not contained in $I$. By \eqref{E:AT=BT}, there is $B \in \A
\cap I$ with $AT = BT$. By \eqref{E:DLI=1}, $B$ centralizes $[D,L]$.
From $C_D(L) \leq C_D(T)$ and \eqref{E:s3based1} it follows that
\[
C_D(A) = C_{[D,L]}(A) \times C_{C_D(L)}(A) = C_{[D,L]}(T) \times (C_D(L) \cap C_D(B)).
\]
Therefore,
\[
|D/C_D(A)| = |C_D(L)/(C_D(L) \cap C_D(B))|\cdot |[D,L]/C_{[D,L]}(T)| = 2\cdot 2 = 4
\]
contrary to $A \in \A$. This contradiction shows that $\A \cap S - \{T\} = \A
\cap I$ and together with \eqref{E:Igen} completes the proof of (S2). Thus, $T
\in \T$. 

Recall that $\B$ consists of the members of $\A$ that are not solitary in
$G$ relative to $S$, and that $J^* \leq J_\B(S)$.  Since $T$ was chosen
arbitrarily subject to $T \leq J$ and $T \nleq I$, we have shown that every
member of $\B$ is contained in $I$; that is, $J_\B(S) \leq I = J \cap J^g$.
Thus, $(J^*)^{g^{-1}} \leq J_\B(S)^{g^{-1}} \leq J$.  Since $J^*$ is weakly
closed in $S$ with respect to $G$ by assumption, it follows that $g \in
N_G(J^*) = H$ and consequently that $J \leq H = H^g$. This contradicts the
choice of $g$ and completes the proof of the lemma.  
\end{proof}

\begin{lemma}\label{L:SSperm}
Let $\T$ be a subset of $\A_D(G)_2 \cap S$ and $\Y = \{[D,T] \mid T \in \T\}$.
Fix $A \in \A_D(G)$ and set $B = C_A(\gen{\Y})$.  Assume each member of $\T$ is
semisolitary relative to a fixed subgroup of $S$ that contains $\gen{\T}$. Then
\begin{enumerate}
\item[(a)] distinct elements of $\T$ have distinct commutators on $D$, and
$\gen{\Y}$ is the direct product of $[D,T]$ for $T \in \T$; and 
\item[(b)] if $A$ acts transitively on $\T$ by conjugation, then either $A =
B$, or $|\T| = |\Y| = 2$, $B$ has index $2$ in $A$, $C_D(A)$ has index
$2$ in $C_D(B)$, and each element of $A-B$ induces a transposition on $\T$. 
\end{enumerate}
\end{lemma}
\begin{proof}
Whenever $T \in \T$, set $Z_T = [D,T]$ for short.  To prove (a),
\begin{eqnarray}\label{E:ijint}
\text{it suffices to show that $Z_T \cap \prod_{R \neq T} Z_R = 1$ for each $T \in \T$.}
\end{eqnarray}
For each $R$ in $\T$, let $W_R$ and $X_R$ be as in Definition~\ref{D:S3based}
in the roles of $W$ and $X$.  For each $R \in \T$,
\[
Z_R = [D, R] = [W_RX_R, R] = [W_R, R] \leq W_R,
\]
so for each $T \in \T$ different from $R$,
\begin{eqnarray}
\label{E:ZjUi}
Z_R = [D,R] = [W_TX_T, R] = [X_T, R] \leq X_T, 
\end{eqnarray}
since $R$ centralizes $W_T$ and normalizes $X_T$ by Definition~\ref{D:S3based}.
Therefore, 
\[
Z_T \cap \prod_{R \neq T} Z_R \leq W_T \cap X_T = 1,
\]
by (SS3), and part (a) now follows from \eqref{E:ijint}. 

Assume $A$ acts transitively on $\T$ by conjugation. Then $A$ acts on $\Y$ in
the same way.  Now $\gen{\Y}$ is a transitive permutation module for $A$ by
part (a), so $|C_{\gen{\Y}}(A)| = 2$.   Suppose $A > B$ and set $m = |A/B|$.
Then $|\Y| = m > 1$ and $|A||C_D(A)| \geq |B||C_D(B)|$, so that 
\begin{eqnarray}\label{E:treq}
m = |A/B| \geq |C_D(B)/C_D(A)| \geq |C_{\gen{\Y}}(B)/C_{\gen{\Y}}(A)| = 2^m/2. 
\end{eqnarray}
Hence $m = 2$ and equality holds in \eqref{E:treq}, so $|C_D(B)/C_D(A)| = 2$.
Since $A \in \A_D(G)$, an element of $A-B$ must act as a transposition on $\Y$
and on $\T$.
\end{proof}

\begin{lemma}\label{L:SSpermlocal}
Fix a subgroup $P \leq S$. Let $\T \subseteq \A_D(G)_2 \cap P$ be the
collection of all subgroups that are semisolitary relative to $P$, and let
$A \in \A_D(G)^\circ \cup \hat{\A}_D(G)^\circ$.  If $A$ normalizes $P$, then
$A$ normalizes every element of $\T$. 
\end{lemma}
\begin{proof}
Set $Z_T = [D,T]$ whenever $T \in \T$, for short.  By assumption $A$ acts on
$\T$. By Lemma~\ref{L:SSperm}(b), each orbit of $A$ on $\T$ has size at most
$2$. We assume that $\{T,R\}$ is a nontrivial orbit and aim for a
contradiction. Set $\Y = \{Z_T, Z_R\}$ and $B = C_A(\gen{\Y})$. Then $B$ has
index $2$ in $A$, and $C_D(A)$ has index $2$ in $C_D(B)$ by
Lemma~\ref{L:SSperm}(b), so that 
\[
|B||C_D(B)| = |A||C_D(A)|. 
\]
Since $A \in \A$, we obtain by minimality of $A$ that $B = 1$ and $|A| = 2$.
Hence $C_D(B) = D$ and
\begin{eqnarray}
\label{E:A=2}
|D/C_D(A)| = 2.
\end{eqnarray}
Since $T$ is semisolitary relative to $P$, we may choose subgroups $W$ and $X$
of $D$ such that $D = W \times X$, $|W| = 4$, $T$ centralizes $X$, $[W,T] = Z_T
= C_W(T)$, and $R$ centralizes $W$. Further, 
\begin{eqnarray}
\label{E:C_W(A)}
C_W(A) > 1
\end{eqnarray}
by \eqref{E:A=2}.

On the other hand, since $A$ transposes $Z_T = C_W(T)$ and $Z_R$, and
since $T$ does not centralize $W$, it follows from \eqref{E:C_W(A)}
that $T$ does not centralize $C_W(A)$. Therefore, $[C_W(A), T] = Z_T$, and so
for the generator $a$ of $A$, 
\[
Z_R = (Z_T)^a = [C_W(A),T^a] = [C_W(A),R] \leq [W,R] = 1,
\]
a contradiction.
\end{proof}

\section{Reduction to the transvection case}\label{S:reduction2}

The objective of this section is to give a proof of
\cite[Proposition~3.2]{Oliver2013} in the case where some minimal offender
under inclusion is not solitary. This result is obtained in
Theorem~\ref{T:ntlim}.

Throughout this section, we fix a finite group $\Gamma$ with Sylow $p$-subgroup
$S$, we set $\F = \F_S(\Gamma)$, and we let $\Q \subseteq \F^c$ be an
$\F$-invariant interval such that $S \in \Q$. In this situation, define
$\Gamma^*$ to be the set of elements of $\Gamma$ that conjugate some member of
$\Q$ into $\Q$.  We say that a $1$-cocycle for the functor $\Z^\Q_\F$ is
\emph{inclusion-normalized} if it sends the class $[\iota_P^Q] \in
\Mor_{\O(\F^c)}(P,Q)$ of any inclusion $\iota_P^Q$ to the identity element of
$Z(P)$ for each $P,Q \in \Q$.  In what follows, we only specify $0$- and
$1$-cochains for the functor $\Z^\Q_\F$ on subgroups in $\Q$, and it is to be
understood that they are the identity on $\F$-centric subgroups outside $\Q$.
Alternatively, apply the isomorphism of cochain complexes in
Lemma~\ref{L:olijm}(a) to view these cochains as restrictions to the full
subcategory of $\O(\F^c)$ with objects in $\Q$.  The reader may wish to recall
the coboundary maps for $0$- and $1$-cochains in our right-handed notation from
\eqref{E:0cob} and \eqref{E:1cob}.

\begin{lemma}\label{L:1cocyclesetup}
Each $1$-cocycle for $\Z^\Q_\F$ is cohomologous to an inclusion-normalized
$1$-cocycle. If $t$ is an inclusion-normalized $1$-cocycle, then 
\begin{enumerate}
\item[(a)] $t([\phi_1]) = t([\phi_2])$ for each commutative diagram
\[
\xymatrix{
P_2 \ar[r]^{\phi_2}  & Q_2\\
P_1 \ar[u]^{\iota_{P_1}^{P_2}} \ar[r]^{\phi_1} & Q_1 \ar[u]_{\iota_{Q_1}^{Q_2}}}
\]
in $\F$ among subgroups in $\Q$;
\item[(b)] the function $\tau\colon \Gamma^* \to
\Gamma^*$ defined by the rule
\[
g^\tau = t([c_g])g,
\]
is a bijection that restricts to the identity map on $S$, and
\[
(g_1g_2\cdots g_n)^{\tau} = g_1^\tau g_2^\tau \cdots g_n^\tau
\]
for each collection of elements $g_i \in \Gamma^*$ with the property that there
is $Q \in \Q$ such that $Q^{g_1\cdots g_i} \in \Q$ for all $1 \leq i \leq n$;
and
\item[(c)] $t = 0$ if and only if $\tau$ is the identity on $\Gamma^*$.  
\end{enumerate}
\end{lemma}
\begin{proof}
Given a $1$-cocycle $t$ for $\Z^\Q_\F$, define a $0$-cochain $u$ by $u(P) =
t([\iota_P^S])$ for each $P \in \Q$. Then for any inclusion $\iota_P^Q$ in $\F$
with $P, Q \in \Q$, we see that
\[
du(P \xrightarrow{[\iota_P^Q)]} Q) = u(Q)u(P)^{-1} = t([\iota_Q^S])t([\iota_P^S])^{-1},
\]
so that
\[
(t\,du)(P \xrightarrow{[\iota_P^Q]} Q) =
t([\iota_P^Q])t([\iota_Q^S])t([\iota_P^S])^{-1} = t([\iota_P^S])t([\iota_P^S])^{-1} = 1
\]
by the $1$-cocycle identity. Hence, $t\,du$ is inclusion-normalized. 

Assume now that $t$ is inclusion-normalized, and let $P_i$, $Q_i$, and $\phi_i$
be as in (a). Since $t$ sends inclusions to the identity, the $1$-cocycle
identity yields
\begin{align*}
t([\phi_1\iota_{Q_1}^{Q_2}]) &= t([\iota_{Q_1}^{Q_2}])^{\phi_1^{-1}}t([\phi_1])
= t([\phi_1]), \text{ and}\\
t([\iota_{P_1}^{P_2}\phi_2]) &=
t([\phi_2])^{(\iota_{P_1}^{P_2})^{-1}}t([\iota_{P_1}^{P_2}]) = t([\phi_2])
\end{align*}
and so (a) follows by commutativity of the diagram. 

Let $\tau$ be given by (b). Since $g \in \Gamma^*$, the conjugation map
$c_g\colon { }^g\!S \cap S \to S \cap S^g$ is a map between subgroups in $\Q$.
Part (a) shows that $t([c_g])$ agrees with the value of $t$ on the class of
each restriction of $c_g$ provided that the source and target of such a
restriction lie in $\Q$.  This shows that $\tau$ is well defined.  Then $\tau$
is a bijection since its inverse is induced by $t^{-1}$ (which is
inclusion-normalized) in the same way.  Further, for $s \in S$, $[c_s] =
[\id_S]$ is the identity in the orbit category, and so $\tau$ is the identity
map on $S$, since $t$ is normalized.

Let $g_1, g_2 \in \Gamma^*$ and $Q \in \Q$ with $Q^{g_1} \leq S$ and
$Q^{g_1g_2} \leq S$. Then by the $1$-cocycle identity,
\[
g_1^\tau g_2^\tau = t([c_{g_1}])g_1\,t([c_{g_2}])g_2 =
t([c_{g_1}])t([c_{g_2}])^{g_1^{-1}}\,g_1g_2 = t([c_{g_1}c_{g_2}])\,g_1g_2 =
(g_1g_2)^\tau. 
\]
Now (b) follows by induction on $n$. Part (c) is clear.
\end{proof}

The function $\tau$ of Lemma~\ref{L:1cocyclesetup}(b) will be called the
\emph{rigid map} associated with the inclusion-normalized $1$-cocycle $t$.

\begin{lemma}\label{L:1cocyclesetup2}
Let $t$ be an inclusion-normalized $1$-cocycle for the functor $\Z_\F^\Q$ and
let $\tau$ be the rigid map associated with $t$. Then
\begin{enumerate}
\item[(a)] for each $Q \in \Q \cap \F^f$ with $C_\Gamma(Q) \leq Q$, there
is $z \in Z(N_S(Q))$ such that $\tau$ is conjugation by $z$ on $N_\Gamma(Q)$; and
\item[(b)] if $z \in Z(S)$ and $u$ is the constant $0$-cochain defined by $u(Q)
= z$ for each $Q \in \Q$, then $du$ is inclusion-normalized and the rigid map
$\upsilon$ associated with $du$ is conjugation by $z$ on $\Gamma^*$.
Conversely, each inclusion-normalized $1$-coboundary $t$ is of the form $t =
du$ for some such constant $0$-cochain $u$.
\end{enumerate}
\end{lemma}
\begin{proof}
We give two proofs for part (a). The first one uses elementary group-theoretic
arguments and the norm map, and is given in Lemma~\ref{L:gaschutz}. The second
is modeled on part of the proof of \cite[Lemma~4.2]{AOV2012} and given now.
From Lemma~\ref{L:1cocyclesetup}(b), $\tau$ induces an automorphism of
$N_\Gamma(Q)$ that is the identity on $N_S(Q)$.  By
\cite[Lemma~1.2]{OliverVentura2009} and its proof, there is a commutative
diagram 
\begin{eqnarray}\label{E:conjbyzseqs}
\vcenter{
\xymatrix{ 1 \ar[r] & Z^1(\Out_\Gamma(Q); Z(Q)) \ar[r]^-{\tilde{\eta}} \ar[d] &
\Aut(N_\Gamma(Q), Q) \ar[r] \ar[d] & \Aut(Q) \ar[d]\\ 1 \ar[r] &
H^1(\Out_\Gamma(Q); Z(Q)) \ar[r]^-{\eta} & \Out(N_\Gamma(Q), Q) \ar[r] &
\Out(Q)} 
}
\end{eqnarray} 
with exact rows, where $\Aut(N_\Gamma(Q), Q)$ is the subgroup of automorphisms
of $N_\Gamma(Q)$ that leave $Q$ invariant, and where $\Out(N_\Gamma(Q), Q)
= \Aut(N_\Gamma(Q),Q)/\Inn(N_\Gamma(Q))$.  Also, $\tilde{\eta}$ maps the
restriction of $t$ (to $\Out_\Gamma(Q)$) to the restriction of $\tau$ to
$N_\Gamma(Q)$. The restriction map $H^1(\Out_\F(Q); Z(Q)) \to
H^1(\Out_S(Q);Z(Q))$ is injective since $\Out_S(Q)$ is a Sylow
$p$-subgroup of $\Out_\F(Q)$ by assumption on $Q$.  Hence $t$ represents the
zero class in $H^1(\Out_\F(Q);Z(Q))$ since $t$ is zero on
$\Out_S(Q)$. It then follows from \eqref{E:conjbyzseqs} that $\tau$
induces an inner automorphism of $N_\Gamma(Q)$. Hence $\tau$ is conjugation by
an element in $Z(N_S(Q))$, since $C_\Gamma(Q) \leq Q$ and $\tau$ is the
identity on $N_S(Q)$.

With $u$ as in (b), we see that $du(P \xrightarrow{[\iota_P^Q]} Q) =
u(Q)u(P)^{-1} = zz^{-1} = 1$, for any inclusion among subgroups when $P \in \Q$
(and when $P \notin \Q$ by \eqref{E:0cob}).  Also, for $g \in \Gamma^*$, 
\[
g^\upsilon = du([c_g])g = z^{-1}z^{g^{-1}}g = g^z,
\]
which proves the first half of (b).

Given an inclusion-normalized $1$-coboundary $t = du$,  set $z = u(S) \in
Z(S)$.  Then for each $Q \in \Q$,
\[
1 = t([\iota_Q^S]) = u(S)u(Q)^{-1} = zu(Q)^{-1},
\]
and so $u$ is constant on $\Q$ with image $z$.
\end{proof}

The following lemma is helpful for showing that a 1-cocycle is trivial in
inductive contexts. It is used in the proofs of Theorem~\ref{T:ntlim} and
Proposition~\ref{P:3.3}. For more information on conjugacy functors,
well-placed subgroups, and conjugation families, please see
Appendix~\ref{A:groups}.  
\begin{lemma}
\label{L:conjfam-rigid}
Let $t$ be an inclusion normalized $1$-cocycle for the functor $\Z_\F^\Q$, and
let $\tau$ be the rigid map corresponding to $t$. Fix a $\Gamma$-conjugacy
functor $W$, and let $\C$ be the associated conjugation family consisting of
the subgroups of $S$ that are well-placed with respect to $W$. Set
\[
\W = \{Q \in \C \cap \Q \mid W(Q) = Q\}, 
\]
and assume that $W(Q) \in \Q$ and $W(W(Q)) = W(Q)$ whenever $Q \in \Q$.  If
$\tau$ is the identity on $N_\Gamma(Q)$ for each $Q \in \W$, then
$\tau$ is the identity on $\Gamma^*$.
\end{lemma}
\begin{proof}
Assume that $\tau$ is the identity on $N_{\Gamma}(Q)$ for each $Q \in \W$.
For each $Q \in \C \cap \Q$, $W(Q)$ is normal in $N_\Gamma(Q)$ by the
definition of a $\Gamma$-conjugacy functor (Definition~\ref{D:conjfunctor}(c)),
and so $N_\Gamma(Q) \leq N_\Gamma(W(Q))$. By the definition of a well-placed
subgroup, $Q \in \C$ implies $W(Q) \in \C$.  Since $W(Q) \in \Q$ and $W(W(Q)) =
W(Q)$ by assumption, we see that $W(Q) \in \W$.  Thus $\tau$ is the identity on
$N_\Gamma(Q)$ for each $Q \in \C \cap \Q$.

It now follows directly from Lemma~\ref{L:1cocyclesetup}(b) that $\tau$ is the
identity on $\Gamma^*$.  We give the details.  Suppose that $\tau$ is the
identity on $N_\Gamma(Q)$ for each $Q \in \C \cap \Q$.  Fix $g \in \Gamma^*$,
and choose $T \in \Q$ with $T^g \leq S$.  By Lemma~\ref{T:wellplaced}, there
are a positive integer $n$, subgroups $Q_1, \dots, Q_n \in \C$, and
elements $g_i \in N_\Gamma(Q_i)$ such that $g = g_1\cdots g_n$, $T \leq Q_1$,
and $T^{g_1\cdots g_i} \leq Q_{i+1}$ for each $i = 1, \dots, n-1$. As $\Q$ is
$\F$-invariant and closed under passing to overgroups, $Q_i \in \Q$ for each
$i$. Since $\tau$ fixes $g_i$ for each $i$ by assumption, $\tau$ fixes $g$ by
Lemma~\ref{L:1cocyclesetup}(b).
\end{proof}

\begin{theorem}\label{T:ntlim}
Let $(\Gamma, S, Y)$ be a reduced setup for the prime $2$. 
Set $D = Z(Y)$, 
$\F = \F_S(\Gamma)$, 
$G = \Gamma/C_\Gamma(D)$,
$\A = \A_D(G)^\circ \cup \hat{\A}_D(G)^{\circ}$,
$\T = \T_D(G)$,
$\B = \A - \T$, and 
\[
\R = \{P \in \sS(S)_{\geq Y} \mid J_\A(P) = Y\}.
\]
Assume $\B$ is not empty. Then $L^2(\F; \R) = 0$. 
\end{theorem}
\begin{proof}
If $\R = \sS(S)_{\geq Y}$, then $L^2(\F;\R) = 0$ by Lemma~\ref{L:olijm}(b) so
we may assume $\Q := \sS(S)_{\geq Y} - \R$ is not empty.  That is, $\A$ is not
empty.  Since $\Q$ is closed under passing to overgroups, $S \in \Q$ and
$J_\A(Q) \in \Q$ for each $Q \in \Q$.   

We will show $L^1(\F;\Q)=0$.  Since $\Q$ and $\R$ are $\F$-invariant intervals
that together satisfy the hypotheses of Lemma~\ref{L:oliexact}, the result then
follows from part (a) of that lemma. 
Fix a $1$-cocycle $t$ for the functor $\Z_\F^{\Q}$. To show that $t$ is
cohomologous to $0$, we may assume by Lemma~\ref{L:1cocyclesetup} that $t$ is
inclusion-normalized.
Let $\tau\colon \Gamma^* \to \Gamma^*$ be the rigid map associated with $t$.

The proof splits into two cases. In Case 1, some member of $\A$ has order at
least $4$. In Case 2, every member of $\A$ has order $2$. We now fix notation
for each case. Use bars to denote images modulo $C_\Gamma(D)$.  If $P$ is a
Sylow $2$-subgroup of a subgroup $H \leq \Gamma$ such that every member of $\A
\cap \bar{P}$ has order $2$, then define $\B(P,H)$ to be the collection of
subgroups in $\A \cap \bar{P}$ that are not solitary in $\bar{H}$ relative to
$\bar{P}$.  In any situation, let $\B(P)$ denote the set of subgroups in $\A
\cap \bar{P}$ that are not semisolitary relative to $\bar{P}$, and set
$\A_{\geq 4} = \{A \in \A \mid |A| \geq 4\}$. 
Define
\begin{align*}
J_1(P) &= J_{\A_{\geq 4}}(P); \text{ and }\\
J_2(P) &= J_{\A_{\geq 4} \cup \B(P)}(P),\\
\end{align*}
so that $Y \leq J_1(P) \leq J_2(P) \leq J_\A(P)$ whenever $P \geq Y$. 
We define two subgroup mappings $W_1$ and $W_2$ on $\sS(S)$, to be employed in
the respective cases.  In all cases, set $W_i(P) = P$ if $P$ does not contain
$Y$. For $P \geq Y$, set
\[
W_1(P) = 
\begin{cases}
J_1(P) &\text{ if $\A_{\geq 4} \cap \bar{P} \neq \varnothing$};\\
J_2(P) &\text{ if $\A_{\geq 4} \cap \bar{P} = \varnothing$ and $\B(P) \neq \varnothing$; and} \\
J_\A(P) &\text{ otherwise,}
\end{cases}
\]
and
\[
W_2(P) = 
\begin{cases} 
J_{\B(S,\Gamma)}(S) &\text{ whenever $J_{\B(S,\Gamma)}(S) \leq P$; and }\\
J_\A(P) &\text{ otherwise.}
\end{cases}
\]
In any case, 
\begin{eqnarray}
\label{E:Wspecial}
\text{$W_i(P) \in \Q$ and $W_i(W_i(P)) = W_i(P)$ whenever $P \in \Q$.}
\end{eqnarray}

Set $W = W_i$ for $i = 1$ or $2$. Then
\begin{eqnarray}
\label{E:Wiscf}
\text{$W$ is a $\Gamma$-conjugacy functor}.
\end{eqnarray}
Indeed, $W(P) \neq 1$ whenever $P \neq 1$ and $W(P) \leq P$ by construction in
each case. In case $W = W_1$, Definition~\ref{D:conjfunctor}(c) holds because
the collections used to define $W$ are $G$-equivariant (e.g., $\B(P^g) =
\B(P)^g$). In case $W = W_2$ and $\A = \A_D(G)_2$,
Definition~\ref{D:conjfunctor}(c) holds since 
\begin{eqnarray}
\label{E:JBSGwc}
\text{$J_{\B(S,\Gamma)}(S)$ is weakly closed in $S$ with respect to $\Gamma$}
\end{eqnarray}
by Lemma~\ref{L:TcapS}. 

Let $\C$ be the collection of subgroups of $S$ that are well-placed with
respect to $W$, and set
\[
\W = \{Q \in \C \cap \Q \mid W(Q) = Q\}.
\]
By \eqref{E:Wspecial}, we are in the situation of Lemma~\ref{L:conjfam-rigid}.
Thus, our strategy is to show that $\tau$ restricts to the identity on
$N_\Gamma(Q)$ for each $Q \in \W$. 

We first arrange that $\tau$ is the identity on $N_\Gamma(W(S))$. Since $W(S)$
is normal in $S$ and contains its centralizer in $\Gamma$, the restriction of
$\tau$ to $N_\Gamma(W(S))$ is conjugation by an element $z \in Z(S)$ by
Lemma~\ref{L:1cocyclesetup2}(a).  Upon replacing $t$ by $t\,du$ where $u$ is
the constant $0$-cochain defined by $u(Q) = z^{-1}$ for each $Q \in \Q$, and
upon replacing $\tau$ by the rigid map associated with $t\,du$, we may assume
by Lemma~\ref{L:1cocyclesetup2}(b) that
\begin{eqnarray}\label{E:tau=1onNJ}
\text{$\tau$ is the identity on $N_\Gamma(W(S))$.} 
\end{eqnarray}

We claim that, to complete the proof of the theorem, it suffices to show that 
\begin{eqnarray}
\label{E:localcontrol}
\parbox[c]{0.80\linewidth}{
\center{
$C_D(H) = C_D(N_H(W(S^*)))$ for all $Q \in \W$, \\
with $H = N_\Gamma(Q)$ and $S^*
= N_S(Q)$.}} 
\end{eqnarray}
To see this, assume \eqref{E:localcontrol}. We show that $\tau$ is the identity
on $N_\Gamma(Q)$ for each $Q \in \W$ by induction on the index of $S^*$ in $S$.
Once this is done, Lemma~\ref{L:conjfam-rigid} shows that $\tau$ is the
identity on $\Gamma^*$, and then $t = 0$ according to
Lemma~\ref{L:1cocyclesetup}(c). 

As before, Lemma~\ref{L:1cocyclesetup2}(a) shows that $\tau$ acts on $H$ as
conjugation by an element $z_H \in Z(S^*) \leq Z(Y) = D$. Assume first that
$S^* = S$. Then $\tau$ is the identity on $N_H(W(S^*))$ by \eqref{E:tau=1onNJ}. 
Hence 
\[
z_H \in C_D(N_H(W(S^*))) = C_D(H)
\]
by \eqref{E:localcontrol}, so that $\tau$ is the identity on $H$ in this case.

Assume now that $S^* < S$. Then $S^* < N_S(W(S^*))$ by
Lemma~\ref{L:conjfunctor}(c), and we see that $\tau$ is the identity on
$N_H(W(S^*))$ by induction. We use this in place of \eqref{E:tau=1onNJ} to
repeat the argument of the last paragraph. Thus, in all cases,
\eqref{E:localcontrol} yields that $\tau$ is the identity on $H$, as desired. 

We are thus reduced to proving \eqref{E:localcontrol}. In each of Case 1 and
Case 2, control of fixed points is shown via the norm argument of
Theorem~\ref{T:glawc2}. In order to check that the hypotheses of
Theorem~\ref{T:glawc2} apply, Lemma~\ref{L:controlbig} or
Lemma~\ref{L:nonS3basednorm} is used depending on which types of offenders are
involved in $S^*$. To help with translation of notation here and in those
results, see Table~\ref{tab:translation}.
\renewcommand{\arraystretch}{2}
\begin{table}
\centering
\begin{tabular}{l|cccc}
\ref{T:glawc2}, \ref{L:controlbig}, \ref{L:nonS3basednorm} & $G$ & $S$ & $J = J_\A(S) = J^*$ & $H$\\
\hline
this proof & $\bar{H} = \overline{N_\Gamma(Q)}$ & $\bar{S}^* =
\overline{N_S(Q)}$ & $\overline{W(S^*)}$ & $\overline{N_H(W(S^*))}$
\end{tabular}
\caption{Translation of notation from \ref{T:glawc2}, \ref{L:controlbig}, and \ref{L:nonS3basednorm}.}
\label{tab:translation}
\end{table}
Theorem~\ref{T:glawc2} is a statement about control of fixed points of $G =
\Gamma/C_\Gamma(D)$ on $D$. As $(\Gamma, S, Y)$ is a reduced setup, each member of
$\Q$ contains $Y = C_S(D)$, so Lemma~\ref{L:normquot} provides the transition
from control of fixed points by normalizers within $G$ and those within
$\Gamma$.  We apply Lemma~\ref{L:normquot} implicitly for this transition in
the arguments that follow.  

Lastly, in order to apply the results of \S\ref{S:norm2}, we need to establish
that 
\begin{eqnarray}
\label{E:W(S*)wc}
\text{$W(S^*)$ is weakly closed in $S^*$ with respect to $H$,}
\end{eqnarray}
and this is done now.  In Case 2, \eqref{E:W(S*)wc} holds by \eqref{E:JBSGwc}.
In Case 1, unless 
\begin{eqnarray}
\label{E:except}
\A_{\geq 4} \cup \B(S^*) = \B(S^*) \neq \varnothing,
\end{eqnarray}
\eqref{E:W(S*)wc} holds since each of the collections used in defining $W_1$
are invariant under $G$-conjugation. In the exceptional situation
\eqref{E:except}, one has that every member of $\A \cap S^*$ is of order $2$
and that $W_1(S^*) = J_{\B(S^*)}(S^*) \leq J_\A(S^*)$.  Now $J_\A(S^*)$ is
abelian by Lemma~\ref{L:nooveroffenders}(c) and assumption, and it is weakly
closed in $S^*$ with respect to $H$. An $H$-conjugate of $J_{\B(S^*)}(S^*)$ in
$S^*$ lies in $J_\A(S^*)$, so the conjugation is induced in $N_H(J_{\A}(S^*))$
by Lemma~\ref{L:burnside}. It follows that this conjugate is generated by
offenders that are semisolitary with respect to $S^*$. This shows $W_1(S^*) =
J_{\B(S^*)}(S^*)$ is weakly closed in $S^*$ with respect to $H$, as claimed,
and we conclude that \eqref{E:W(S*)wc} holds in all cases.

We now prove in each of the two cases that $N_H(W(S^*))$ controls the fixed
points of $H$ on $D$; that is, \eqref{E:localcontrol} holds.

\medskip
\noindent
\textbf{Case 1:} Some member of $\A$ has order at least $4$. 

\smallskip
Put $W = W_1$.  Assume first that $S^* = S$. Since Case 1 holds, the collection
$\A_{\geq 4} = \A_D(G)^\circ_{\geq 4} \cup \hat{\A}_D(G)^{\circ}$ is not empty.
Hence $W(S^*) = J_1(S)$ by definition of $W$.  By Lemma~\ref{L:controlbig}, the
hypotheses of Theorem~\ref{T:glawc2} apply, and thus $N_H(W(S^*))$ controls
the fixed points of $H$ on $D$. 

Assume now that $S^* < S$.  If $\A_{\geq 4} \cap \bar{S}^*$ is not empty, then
$W(S^*) = J_1(S^*)$, and $N_H(W(S^*))$ controls the fixed points of $H$ on $D$
via Lemma~\ref{L:controlbig} as before. 

Assume that $\A_{\geq 4} \cap \bar{S^*}$ is empty but that $\B(S^*)$ is not
empty. Then every member of $\A \cap \bar{S^{*}}$ is of order $2$, and some
member of $\A \cap \bar{S^*}$ is not semisolitary relative to $\bar{S^*}$. Such
an offender is not solitary in $\bar{H}$ with respect to $\bar{S^*}$ by
Remark~\ref{R:S3basedimpliesS3similar}. Therefore, Lemma~\ref{L:nonS3basednorm}
applies, which yields by Theorem~\ref{T:glawc2} that $N_H(W(S^*))$ controls the
fixed points of $H$ on $D$. 



Finally, assume that $\A_{\geq 4} \cap \bar{S^*}$ and $\B(S^*)$ are both empty. We
will show this leads to a contradiction -- this is a critical step in the
proof.  By definition of $W_1$ and our choice of $Q \in \W$, we have 
\begin{eqnarray}\label{E:something}
Q = W(Q) = J_\A(Q) \leq J_\A(S^*) = W(S^*).
\end{eqnarray}
As every member of $\A \cap \bar{S^*}$ is of order $2$, 
\begin{eqnarray}
\text{$\ol{J_\A(S^*)}$ is elementary abelian}  
\end{eqnarray}
by Lemma~\ref{L:nooveroffenders}(c). 

As Case 1 holds and each member of $\A \cap \bar{S^*}$ is semisolitary relative
to $\bar{S^*}$, we have $J_{\A}(S^*) < J_\A(S)$. Since  $J_\A(S^*) \leq
J_{\A}(N_S(J_\A(S^*)))$, this inclusion therefore must be strict by
Lemma~\ref{L:conjfunctor}(b). Choose $A \leq S$ with $\bar{A} \in \A \cap
\bar{S}$ such that 
\begin{eqnarray}\label{E:thisthing} 
A \leq J_\A(N_S(J_\A(S^*))), \text{ but } A \nleq J_\A(S^*).
\end{eqnarray}
It follows from the definitions that each subgroup semisolitary relative to
$\bar{S^*}$ is also semisolitary relative to $\overline{J_\A(S^*)}$.  As $A$
normalizes $J_\A(S^*)$
we are thus in the situation of Lemma~\ref{L:SSpermlocal} with
$\ol{J_\A(S^*)}$ in the role of $P$, and $\A \cap \bar{S^*}$ in the role
of $\T$ there.  By that lemma, $\bar{A}$ normalizes every member of $\A \cap
\bar{S^*}$. Thus $A$ normalizes each of their preimages in $S$.  However $Q$ is
generated by the preimages of a subset of $\A \cap \bar{S^*}$ by
\eqref{E:something}, and hence $A$ normalizes $Q$. But then $A \leq J_\A(S^*)$,
contrary to \eqref{E:thisthing}. This contradiction completes the proof of Case
1. 

\medskip
\noindent
\textbf{Case 2:} Each member of $\A$ is of order $2$. 

\smallskip
Put $W = W_2$ and assume first that $S^* = S$. We verify the hypotheses of
Lemma~\ref{L:nonS3basednorm}. By assumption, $\A = \A_D(G)_2$, and so
$\hat{\A}_D(G)$ is empty by Remark~\ref{R:Ahat4}. By definition of $W_2$, we
have that $W(S) = J_{\B(S,\Gamma)}(S)$ in the role of $J^*$ of
Lemma~\ref{L:nonS3basednorm}. It follows from Definition~\ref{D:S3based} that
every subgroup solitary in $\bar{H}$ relative to $\bar{S}$ is also solitary in
$G$ relative to $\bar{S}$. Therefore $\B(S,\Gamma) \subseteq \B(S,H)$, where
the former is not empty by assumption, and where the latter is in the role of
$\B$ of Lemma~\ref{L:nonS3basednorm}.  Hence we may apply
Lemma~\ref{L:nonS3basednorm} to obtain the hypotheses of Theorem~\ref{T:glawc2}
which yields that $N_H(W(S))$ controls the fixed points of $H$ on $D$.



Finally, assume that $S^* < S$.  Since Case 2 holds,  $\ol{J_\A(S)}$ is
elementary abelian by Lemma~\ref{L:nooveroffenders}(c). Therefore as $Q
\in \W$, 
\[
Q = W(Q) \leq J_\A(Q) \leq J_\A(S),
\]
and $\ol{J_\A(S)}$ centralizes $\bar{Q}$. Hence $J_\A(S) \leq S^*$ so that
in particular,
\[
S^* \geq J_{\B(S,\Gamma)}(S).
\]
This shows that $W(S^*) = J_{\B(S,\Gamma)}(S)$ by definition of $W_2$.  As in
the situation where $S^* = S$, we may apply Lemma~\ref{L:nonS3basednorm} to
obtain the hypotheses of Theorem~\ref{T:glawc2}, which yields that
$N_H(W(S^*))$ controls the fixed points of $H$ on $D$ as before. This concludes the
proof in Case 2. Therefore, \eqref{E:localcontrol} holds in all cases, which
completes the proof of the theorem.
\end{proof}

\section{Transvections}\label{S:transvections}

The aim of this section is to give a proof, in Proposition~\ref{P:3.3},
of Proposition~3.3 of \cite{Oliver2013} for $p=2$. This result and the proof of
\cite[Theorem~3.4]{Oliver2013} give Theorem~\ref{T:main} when $p=2$. 

Using McLaughlin's classification of irreducible subgroups of $SL_n(2)$
generated by transvections, we first classify in Theorem~\ref{T:genbyS3based}
those finite groups which have no nontrivial normal $2$-subgroups and are
generated by solitary offenders.  Recall that by a \emph{natural} $S_m$-module
($m \geq 3$), we mean the nontrivial composition factor of the standard
permutation for $S_m$ over the field with two elements.

\begin{lemma}\label{L:soltrans}
Let $G$ be a finite group acting irreducibly on an elementary abelian $2$-group
$W$. Assume that $G$ is generated by transvections. Then $\T_W(G)$ is not empty
if and only if $G$ is isomorphic to a symmetric group of odd degree and $W$ is
a natural module for $G$. Moreover, in this case, $\T_W(G)$ is the set of
transpositions.
\end{lemma}
\begin{proof}
Assume first that $G$ is generated by transvections on the irreducible module
$W$. By a result of McLaughlin \cite{McLaughlin1969}, $G$ is isomorphic to
$SL(W)$, or the dimension $n$ of $W$ is even and at least $4$ and $G$ is
isomorphic to $Sp(W)$, $O^{-}(W)$, $O^+(W)$, $S_{n+1}$, or $S_{n+2}$. For the
classical groups, $\A_W(G)_2$ is the set of transvections; for the symmetric
groups, $\A_W(G)_2$ is the set of transpositions. In all cases, $\A_W(G)_2$ is
a single $G$-conjugacy class.

Fix a Sylow $2$-subgroup $S$ of $G$, and assume first that $G = SL(W)$ with $n
\geq 3$. Since $\A_W(G)_2$ is a single conjugacy class, either $\T_W(G)$ is
empty or $\T_W(G) = \A_W(G)_2$.  Since $S$ is itself generated by
transvections, (S2) forces $S$ to be abelian in the latter case, a
contradiction.

Assume that $G$ is a symmetric group of degree $n+2 \geq 6$. We may assume $S$
stabilizes the partition $\{\{1,2\},\dots, \{n+1, n+2\}\}$, and then
$\gen{\A_W(G)_2 \cap S}$ is the centralizer of this partition. Fix $A =
\gen{(2j-1,2j)} \in \A_W(G)_2 \cap S$, and let $L \cong S_3$ be a subgroup of
$G$ containing $A$.  Then all members of $\A_W(G)_2 \cap L$ are conjugate, and
so the support of $L$ is a three-element set, say $\{2j-1, 2j, k\}$. Hence $L$
does not centralize the element of $\A_W(G)_2 \cap S$ moving $k$, and thus $A$
is not solitary in $G$ relative to $S$. 

Assume that $G = Sp(W)$ preserves the symplectic form $\mathfrak{b}$. Fix a
maximal isotropic subspace $W_0$ stabilized by $S$, and let $U$ be the
unipotent radical of its stabilizer in $G$. Then all members of $\A_W(G)_2 \cap
S$ are contained in $U$.  Let $A$ be one of them, having center $\gen{e}
\subseteq W_0$, and let $L \cong S_3$ be a subgroup of $G$ containing $A$.
Since two symplectic transvections commute if and only if their centers are
orthogonal with respect to $\mathfrak{b}$, $[W,L]$ is a hyperbolic line
containing $e$. Since $n \geq 4$, we may find $e' \in [W,L]^\perp \cap
W_0$, and then $L$ does not centralize the member of $\A_W(G)\cap S$ with
center $\gen{e+e'}$. So $A$ is not solitary.

Assume that $G$ is an orthogonal group preserving the quadratic form
$\mathfrak{q}$ with associated symplectic form $\mathfrak{b}$. If $n=4$, then
since $G$ is generated by transvections, $G = O_4^-(2) \cong S_5$.  Thus, we
may assume that $n \geq 6$. Choose a maximal isotropic subspace $W_0$ (with
respect to $\mathfrak{b}$) stabilized by $S$ and such that $W_0$ contains a
nonsingular vector. Let $U$ be the unipotent radical of the stabilizer of the
radical of $\mathfrak{q}|_{W_0}$. Fix a nonsingular vector $e \in W_0$, let $A
\leq U$ be generated by the transvection with center $\gen{e}$, and let $L
\cong S_3$ be a subgroup of $G$ such that $L$ contains $A$.  As before, the
restriction of $\mathfrak{b}$ to $[W,L]$ is nondegenerate. Since $\dim(W_0)
\geq 3$, there is a singular vector $e'$ in $[W,L]^\perp \cap W_0$, and then
$L$ does not centralize the orthogonal transvection with center $\gen{e+e'}$.
So $A$ is not solitary.

Therefore, $G$ is a symmetric group of odd degree and $W$ is a natural module
for $G$. 

For the converse, let $G = S_{2n+1}$ ($n\geq 1)$, $\Omega$ the standard
$G$-set, and identify $W$ with the set of even order subsets of $\Omega$.  To
show that each transposition generates a solitary subgroup, we may restrict our
attention to $T = \gen{(2n-1,2n)}$. Consider the partition $\{\{2i-1,2i\} \mid
1\leq i \leq n\}$ of $\Omega-\{2n+1\}$, and let $S$ be a Sylow $2$-subgroup of
$G$ stabilizing this partition. Then $\A_W(G)_2 \cap S =
\{\gen{(1,2)},\dots,\gen{(2n-1,2n)}\}$. Hence, taking $L$ be the symmetric
group induced on $\{2n-1,2n,2n+1\}$, one easily checks that (S1)-(S3) hold in
Definition~\ref{D:S3based}, so that $T$ is solitary in $G$ relative to $S$.
\end{proof}

\begin{theorem}\label{T:genbyS3based}
Let $G$ be a finite group, let $D$ be an abelian $2$-group on which $G$ acts
faithfully, and set $\T = \T_D(G)$. Assume that $O_2(G) = 1$ and that $G =
\gen{\T}$ is generated by its solitary offenders.  Then there exist a positive
integer $r$ and subgroups $E_1$, \dots, $E_r$ such that
\begin{enumerate}
\item[(a)] $G = E_1 \times \cdots \times E_r$, $\T = (\T \cap E_1) \cup \cdots
\cup (\T \cap E_r)$, and $E_i \cong S_{m_i}$ with $m_i$ odd for each $i$; and
\item[(b)] $D = V_1 \times \cdots \times V_r \times C_D(G)$ with $V_i =
[D,E_i]$ a natural $S_{m_i}$-module, and with
$[V_i, E_j] = 1$ for $j \neq i$.
\end{enumerate}
\end{theorem}

\begin{proof}
Set $W=[D,G]C_D(G)/C_D(G)$, and let $(G,D)$ be a counterexample with $|G|+|D|$
minimal.  Then $[D,G] = \gen{[D,T] \mid T \in \T}$ is elementary abelian since
each $[D,T]$ has order $2$, and hence $W$ is elementary abelian.
We will show in Step 1 that $[D,G,G] = [D,G]$, in Step 2 that
$C_{D/C_D(G)}(G) = 1$, in Step 3 that $G$ is faithful on $W$ with $W = [W,G]$
and $C_W(G) = 1$, in Step 4 that the theorem for $(G,W)$ implies the theorem
for $(G,D)$, and in Step 5 that $W$ is irreducible. The theorem then follows
from Lemma~\ref{L:soltrans}. We prefer to give essentially complete proofs from
first principles.

When $D > D_1 > \dots > D_k > 1$ is a chain of $G$-invariant subgroups of $D$
(or of any other abelian $p$-group with faithful action from $G$) we say for
short that a subgroup $K$ of $G$ acts \emph{nilpotently} on the chain if $K$
acts trivially on successive quotients.  If a subgroup $K$ acts nilpotently on
such a chain and $k \geq 1$, then $K$ is contained in the subgroup $L$ of all
elements that act trivially on successive quotients.  Clearly, $L$ is normal in
$G$, and $L$ is a $2$-group by \cite[Theorem~5.3.3]{Gorenstein1980}. Hence, $K
\leq L \leq O_2(G)$ in this case.

\smallskip
\noindent
{\bf Step 1:} For each $T \in \T$, choose a subgroup $L$ of $G$ such that $L
\geq T$ and $L \cong S_3$ as in Definition~\ref{D:S3based}.  Then $[D,T] \leq
[D,L] \leq [D,G]$ and $[D,L,L] = [D,L]$ by Lemma~\ref{L:nota2group}, so
that $[D,T] \leq [D,L,L] \leq [D,G,G]$. Hence $T$ centralizes $G/[D,G,G]$.  It
follows that $G$ centralizes $G/[D,G,G]$ since $G = \gen{\T}$ and the choice of
$T$ was arbitrary.  That is, $[D,G] \leq [D,G,G]$.  Since the reverse inclusion
holds, we conclude that $[D,G] = [D,G,G]$.

\smallskip
\noindent
{\bf Step 2:} Let $D_1$ be the preimage of $C_{D/C_D(G)}(G)$ in $D$. Suppose
that $D_1 > C_D(G)$. Then $C_D(G) > 1$. As before fix $T \in \T$ and choose $L$
as in Definition~\ref{D:S3based} for $T$. Then $O^2(L)$ acts nilpotently on the
chain $D_1 > C_D(G) > 1$. Since $O^2(L)$ is of odd order, it centralizes $D_1$.
Hence, by (S3),
\[
D_1 \leq C_D(O^2(L)) = C_D(L) \leq C_D(T).
\]
That is, $T$ centralizes $D_1$.  We conclude that $G$ centralizes $D_1$ since
$G = \gen{\T}$, and since the choice of $T$ was arbitrary. This contradicts
$D_1 > C_D(G)$.  Therefore, $D_1 = C_D(G)$.

\smallskip
\noindent
{\bf Step 3:} Set $W = [D,G]C_D(G)/C_D(G)$ as above. 

Using faithfulness of $G$ on $D$ and the assumption $O_2(G) = 1$, we see that
\begin{eqnarray}\label{E:Wneq1}
W > 1
\end{eqnarray}
since otherwise $G$ acts nilpotently on $D > C_D(G) > 1$.  Let $K$ be the
kernel of the action of $G$ on $W$. Then $K$ acts nilpotently on the chain $D
\geq [D,G]C_D(G) > C_D(G) \geq 1$, with the strict inclusion from
\eqref{E:Wneq1}, and hence $K \leq O_2(G) = 1$.
We conclude that
\begin{eqnarray}\label{E:Wfaithful}
\text{$G$ is faithful on $W$}.
\end{eqnarray}
By Steps 1 and 2, 
\begin{eqnarray}\label{E:noWfps}
W = [W,G] \quad \text{ and } \quad C_W(G) = 1.
\end{eqnarray}

\smallskip
\noindent
{\bf Step 4:} Suppose that the theorem holds for $(G,W)$ with respect to
$\T_\W(G)$. In particular, $G$ is a direct product of symmetric groups $E_i$ of
odd degree $m_i$ and $W$ is a direct sum of natural modules $W_i$ for $E_i$ ($1
\leq i \leq r$) satisfying $[W_i,E_j] = 1$ whenever $j \neq i$.  Also, since
part (a) holds, each member of $\T_W(G)$ is contained in $E_i$ for some $i$,
and so is generated by a transposition by Lemma~\ref{L:soltrans}.  Each member
of $\T$ is faithful on $W$ by Step 3, and so has commutator of order $2$ there.
Thus, $\T \subseteq \T_W(G)$.  For each $i$, some element of $\T$ lies in
$E_i$, because $\T$ generates $G$; hence, every transposition in $E_i$ lies in
$\T$ because $\T$ is invariant under $G$-conjugation. Consequently, $\T =
\T_W(G)$, so that (a) holds for $(G,D)$.

Set $V_1 = [D, E_1]$.  Then under the projection from $D$ onto $D/C_D(G)$,
the image of $V_1$ contains $[W,E_1] = W_1$, and so $|V_1| \geq |W_1|$.
On the other hand, we may choose $m_{1}-1$ elements $T_{1},\dots,T_{m_1-1} \in
\T \cap E_1$ corresponding to adjacent transpositions that generate $E_1$, and
see that 
\[
|V_1| = |\gen{[D,T_1],\dots,[D,T_{m_1-1}]}| \leq \prod_{1 \leq i \leq m_1-1}
|[D,T_i]| = 2^{m_1-1} = |W_1|,
\]
since $|[D,T_i]| = 2$ for each $1 \leq i \leq m_1-1$.  Hence $V_1 \cong W_1$ is
a natural $E_1$-module. It follows that $C_D(E_1) \cap V_1 = 1$. Moreover,
\[
|D/C_D(E_1)| = |D/(\cap_i C_D(T_i))| \leq \prod_{1 \leq i \leq m_1-1} |D/C_D(T_i)| = 2^{m_1-1} = |V_1|,
\]
since $|D/C_D(T_i)| = 2$ for each $i$. We conclude that $D = V_1 \times
C_D(E_1)$. In the case that $r=1$, this shows that (b) holds for $G$ and $D$.
Otherwise, apply induction (on $r$) to $E_2\cdots E_r$, $\T \cap E_2\cdots
E_r$, and $C_D(E_1)$ to obtain
\[
D = V_1 \times \cdots \times V_r \times C_D(E_1\cdots E_r)
\]
which yields part (b) for $(G,D)$.

\smallskip
\noindent
{\bf Step 5:}
By Step 4 and induction, $D = W$. We next show that $W$ is irreducible
for $G$.  Assume on the contrary that $W_1$ is a nontrivial proper
$G$-invariant subgroup of $W$. Set
\begin{align*}
\T_1 &= \{T \in \T \mid [W,T] \leq W_1\},\\
\T_2 &= \{T \in \T \mid [W_1,T] = 1\},\\
\end{align*}
$G_1 = \gen{\T_1}$, and $G_2 = \gen{\T_2}$. Then $G_1$ and $G_2$ are normal in $G$. 

Let $T \in \T-\T_2$. Then
\[
1 < [W_1,T] \leq [W,T] = [D,T],
\]
so as $|[D,T]| = 2$, all these inclusions are equalities. In particular, $[W,T]
= [W_1, T] \leq W_1$, which yields $T \in \T_1$. Hence
\begin{eqnarray}\label{E:union}
\T = \T_1 \cup \T_2 \quad \text{ and } \quad G = G_1G_2. 
\end{eqnarray}

Set $K = C_G(W/W_1) \cap C_G(W_1)$. Then $K$ acts nilpotently on the chain $W >
W_1 > 1$, and so $K \leq O_2(G) = 1$ by \eqref{E:Wfaithful} and assumption on
$G$. Since $[G_1,G_2] \leq G_1 \cap G_2 \leq K$, we see that 
\begin{eqnarray}\label{E:intersection}
\T_1 \cap \T_2 = \varnothing \quad \text{ and } \quad G = G_1 \times G_2
\end{eqnarray}
from \eqref{E:union}.

Now $[W,G_1,G_2] = 1$ by construction and $[G_1,G_2,W] = 1$ from
\eqref{E:intersection}, so $[W,G_2,G_1]=1$ by the Three Subgroups Lemma.  Hence
$[W,G_1G_2] = [W,G_1][W,G_2]$. Further, $[W,G_1] \cap [W,G_2] \leq
C_W(G_1G_2)$, which is the identity by \eqref{E:noWfps}, and so
\begin{eqnarray}\label{E:Wdecomp}
W = [W,G] = [W,G_1] \times [W,G_2]
\end{eqnarray}
again by \eqref{E:noWfps} and \eqref{E:union}.

Finally, $\T_1$ is not empty since otherwise $G = G_2$ centralizes $W_1$
contrary to \eqref{E:noWfps}. Similarly, $\T_2$ is not empty because $W =
[W,G]$ by \eqref{E:Wdecomp}. Hence $1 < |G_1| < |G|$ and $1 < |G_2| < |G|$. One
then checks that $T \in \T_k$ is solitary in $G_k$ (on $W_k$, $k = 1, 2$),
using \eqref{E:Wdecomp} and the fact that an $L \cong S_3$ containing $T$ in
$G$ is generated by $G$-conjugates of $T$.

Induction applied to $(G_1,W_1)$ and $(G_2,W_2)$ now yields the theorem for
$(G,W) = (G,D)$. We conclude that $W$ is irreducible for $G$, as desired.
In particular, each element of $\T$ induces a transvection on $W$.

Now Lemma~\ref{L:soltrans} shows that $G$ is a symmetric group of odd degree
and $W$ is a natural module for $G$, and $\T$ is the collection of subgroups
generated by a transposition. Hence $(G,W) = (G,D)$ is not a
counterexample.
\end{proof}

The following is \cite[Proposition~3.3]{Oliver2013} for $p=2$. The verification
of \eqref{E:r=1} in the proof is inspired by Lemma~7.7 of \cite{Chermak2013}.

\begin{proposition}\label{P:3.3}
Let $(\Gamma,S,Y)$ be a general setup for the prime $2$. Set $\F =
\F_S(\Gamma)$, $D = Z(Y)$, and $G = \Gamma/C_\Gamma(D)$.  Let $\R \subseteq
\sS(S)_{\geq Y}$ be an $\F$-invariant interval such that for each $Q \in
\sS(S)_{\geq Y}$, $Q \in \R$ if and only if $J_{\A_D(G)}(Q) \in \R$.  Then
$L^k(\F;\R) = 0$ for all $k \geq 2$. 
\end{proposition}
\begin{proof}
Assume the hypotheses of the proposition, but assume that the conclusion is
false. Let $(\Gamma, S, Y, \R, k)$ be counterexample for which the four-tuple
$(k,|\Gamma|,|\Gamma/Y|,|\R|)$ is minimal in the lexicographic ordering.  Steps
1--3 in the proof of \cite[Proposition~3.3]{Oliver2013} show that $\R = \{P
\leq S \mid J_{\A_D(G)}(P) = Y\}$, $k = 2$, and $(\Gamma,S,Y)$ is a reduced
setup. 

By Lemma~\ref{L:olijm}(b), $\R$ is a proper subset of $\sS(S)_{\geq Y}$.  Let
$\Q = \sS(S)_{\geq Y}-\R$ and $\A = \A_D(G)^\circ \cup \hat{\A}_D(G)^\circ$.
Then $\Q$ is not empty, and so $\A$ is not empty.  Since $(\Gamma, S, Y)$ is a
counterexample, Theorem~\ref{T:ntlim} shows that 
\begin{eqnarray}\label{E:transvections} 
\A = \T_D(G).  
\end{eqnarray}
That is, $\hat{\A}_D(G)$ is empty and every best offender minimal under
inclusion is solitary in $G$ relative to $\bar{S}$. 
Since $L^2(\F; \R) \neq 0$, we see that 
\begin{eqnarray}\label{E:L1Q}
L^1(\F,\Q) \neq 0.
\end{eqnarray}
from Lemma~\ref{L:oliexact}.  We prove next that
\begin{eqnarray}\label{L:genbytrans}
G = \gen{\A}. 
\end{eqnarray}
Let $G_0 = \gen{\A}$. Let $\Gamma_0$ be the preimage of $G_0$ in $\Gamma$, set
$S_0 = S \cap \Gamma_0$, set $\F_0 = \F_{S_0}(G_0)$, and set $\Q_0 =
\sS(S_0)_{\geq Y} \cap \Q$.  Then $\Gamma_0 \norm \Gamma$ and by
\eqref{E:transvections}, we have $Y \leq S_0$. Further, $(\Gamma_0, Y, S_0)$ is
a reduced setup and $\Q_0$ is an $\F_0$-invariant interval. Since each member
of $\A$ is contained in $G_0$, we have $\Gamma_0 \cap Q \in \Q$ for each $Q \in
\Q$. By Lemma~\ref{L:olirestinj}, the restriction map induces an injection
$L^1(\F;\Q) \to L^1(\F_0; \Q_0)$ and so $L^1(\F_0; \Q_0) \neq 0$ by
\ref{E:L1Q}. Therefore, $\Gamma = \Gamma_0$ (and hence $G = G_0$) by minimality
of $|\Gamma|$, which completes the proof of \eqref{L:genbytrans}.

Therefore, $G$ and its action on $D$ are described by
Theorem~\ref{T:genbyS3based}.  We adopt the notation in that
theorem for the remainder of the proof.  In the decomposition of part (b)
there, each $V_i$ is $G$-invariant and so each is $S$-invariant.  Thus, the
centralizer of $S$ in $D$ factors as
\begin{eqnarray}\label{E:CDSfact}
C_D(S) = C_{V_1}(S) \times \cdots \times C_{V_r}(S) \times C_D(G). 
\end{eqnarray}

Fix an inclusion-normalized $1$-cocycle $t$ for $\Z_\F^\Q$ representing a
nonzero class in $L^1(\F;\Q)$ by \eqref{E:L1Q}, and let $\tau\colon \Gamma^* \to
\Gamma^*$ be the rigid map associated with $t$.  We show next that
\begin{eqnarray}\label{E:r=1}
r = 1. 
\end{eqnarray}
We assume $r > 1$ and aim for a contradiction.  Let $G_1 = E_1$ and $G_2 =
E_2\cdots E_r$.  For $i = 1, 2$, let $K_i$ be the preimage of $G_i$ in
$\Gamma$, and set $\Gamma_i = K_iS$ and $\F_i = \F_S(\Gamma_i)$. Then
$(\Gamma_i, S, Y)$ is a general setup, and $\Gamma_i < \Gamma$ by assumption.
Hence $L^1(\F_i;\Q) \cong L^2(\F_i;\R) = 0$ by minimality of $|\Gamma|$.  As
the restriction of $t$ to $\O(\F_i^c)$ represents the zero class, by
Lemma~\ref{L:1cocyclesetup2}(b) there are elements $z_i \in Z(S) = C_D(S)$ such
that $\tau$ is conjugation by $z_i$ on $(\Gamma_i)^*$.  Since $D = [D,G_i]
\times C_D(G_i)$ by Theorem~\ref{T:genbyS3based}(b) and $C_D(S)$
factorizes correspondingly, it is also the case that $\tau$ is conjugation by
the component of $z_i$ in $[D,G_i]$ on $(\Gamma_i)^*$, and so we may and will
take $z_i \in [D, G_i]$ to be this component. Set $t' = t\,du$ where $u$ is the
constant $0$-cochain defined by $u(P) = (z_1z_2)^{-1}$ for each $P \in \Q$.
Then by Lemma~\ref{L:1cocyclesetup2}(b), upon replacing $t$ by $t'$ and $\tau$
by the rigid map $\tau'$ associated with $t'$, we may assume that $\tau$ is the
identity when restricted to $(\Gamma_i)^*$ for each $i = 1, 2$. 

Now the objective is to show that $\tau$ is the identity on $\Gamma^*$.  Let
$W$ be the $\Gamma$-conjugacy functor defined by $W(P) = J_\A(P)$ for each $P
\geq Y$, and by $W(P) = P$ otherwise.  Let $\W$ be the collection of subgroups
$Q \in \Q$ such that $W(Q) = Q$ and such that $Q$ is well-placed with respect
to $W$.  (We refer to Appendix~\ref{A:groups} for background on conjugacy
functors and well-placed subgroups.) Since $W(W(Q)) = W(Q)$ and $W(Q) \in \Q$
whenever $Q \in \Q$, it thus suffices to show that $\tau$ is the identity on
$N_\Gamma(Q)$ for each $Q \in \W$ by Lemma~\ref{L:conjfam-rigid}.

Let $Q \in \W$, so that the image of $Q$ in $G$ is generated by members
of $\A$. It then follows from Theorem~\ref{T:genbyS3based}(a) that $Q =
Q_1Q_2$ with $Q_1 \cap Q_2 = Y$, where $\bar{Q}_1$ and $\bar{Q}_2$ are the
projections in $G_1$ and $G_2$ of $\bar{Q}$.  If it happens that $Q_i \notin
\Q$ for $i = 1$ or $2$, this means that $Q_i = Y$. Since $Y \notin \Q$ but $Q
\in \Q$, we may assume without loss that $Q_2 \in \Q$. Now $Q$ is well-placed
and $W(Q) = Q$, so that $Q$ is fully $\F$-normalized. A straightforward
argument shows that $Q_2$ is also fully $\F$-normalized.

Let $g \in N_\Gamma(Q_2)$, and write $g = g_1g_2$ with $g_i \in K_i$.  Since
$N_S(Q_2)$ is a Sylow $2$-subgroup of $N_\Gamma(Q_2)$, we have by
Lemma~\ref{L:normquot} that $\ol{N_\Gamma(Q_2)} = N_G(\bar{Q}_2)$, and the
latter is $G_1 \times N_{G_2}(\bar{Q}_2)$.  Since
$\bar{g}_1 \in G_1$ normalizes $\bar{Q}_2$, some element in the coset
$g_1C_\Gamma(D)$ normalizes $Q_2$. We may therefore write $g_1 = h_1c_1$ where
$h_1 \in N_{\Gamma_1}(Q_2)$ and $c_1 \in C_\Gamma(D)$.  So as
$h_1 \in (\Gamma_1)^*$ and $c_1g_2 \in (\Gamma_2)^*$ (both conjugate $Q_2$ to
$Q_2$), we see that $\tau$ fixes $g$. We conclude that $\tau$ is the identity
on $N_\Gamma(Q_2)$. However, then $\tau$ is the identity on $N_\Gamma(Q)$ since
$N_\Gamma(Q) \leq N_\Gamma(Q_2)$.  We conclude that $\tau$ is the identity on
$\Gamma^*$, a contradiction.  This concludes the proof of \eqref{E:r=1}.

By \eqref{E:r=1}, we may fix $m = 2n+1$ such that $G = E_1 \cong S_m$ and write
$\Omega$ for the set of even order subsets of $\{1,\dots,m\}$. Identify $G$
with $S_m$ and $V_1$ with $\Omega$, and set $V = V_1$ for short. We may assume
that $S$ stabilizes the collection $\{\{2i-1,2i\} \mid 1 \leq i \leq n\}$.  In
this situation, whenever $\tau$ is conjugation by an element of $C_D(S)$ on
some subset of $\Gamma$, we always take that element to lie in $V$ by
convention; since $D = [D,G] \times C_D(G) = V \times C_D(G)$ and $C_D(G) \leq
C_D(S)$, there is no loss in doing this.

For each $1 \leq i \leq n$, set $z_i = \{1,\dots,2i\}$ and $z'_i =
\{2i+1,\dots,2n\}$, and let $Q_i$ be the preimage in $S$ of $\gen{(1,2),\dots,
(2i-1,2i)}$. Then $C_V(N_\Gamma(Q_n)) = \gen{z_n}$, and
\begin{eqnarray}\label{E:N_Gamma(Qi)}
C_V(N_\Gamma(Q_i)) = \gen{z_i} < \gen{z_i, z'_i} =  C_V(N_\Gamma(Q_i) \cap
N_\Gamma(Q_n)) \quad \text{ and } \quad z_n = z_iz_i' 
\end{eqnarray}
for all $1 \leq i \leq n-1$. 

Set $\Gamma_n = C_\Gamma(z_n)$ (the setwise stabilizer of $z_n$).  Then $S \leq
N_\Gamma(Q_n) \leq \Gamma_n$, and the image of $\Gamma_n$ in $G$ is isomorphic
with $S_{2n}$.  Since $(\Gamma_n, S, Y)$ is a general setup and not a
counterexample, we may adjust $t$ by a coboundary and assume that 
\begin{eqnarray}\label{E:tauGamman}
\tau \text { is the identity on } \Gamma_n^* := \Gamma_n \cap \Gamma^*. 
\end{eqnarray}

With notation and argument as in the proof of \eqref{E:r=1}, it suffices to
show that $\tau$ is the identity on $N_\Gamma(Q)$ for each $Q \in \W$.
Moreover, notice that for any subgroup $Q \in \Q$ with $W(Q) = Q$, we have
$W(S^*) = Q_n$ for every Sylow $2$-subgroup $S^*$ of $N_\Gamma(Q)$. Hence, if
$W(Q) = Q$, then $Q \in \W$ if and only if $Q$ is fully $\F$-normalized.

Suppose it can be shown that
\begin{eqnarray}
\label{E:tau=1onNQi}
\text{$\tau$ is the identity on $N_\Gamma(Q_i)$ for
each $1 \leq i \leq n$,}
\end{eqnarray}
and fix $Q \in \W$. Then, by Lemma~\ref{L:burnside}, there is an integer $i$
and an element $g \in N_\Gamma(Q_n) \subseteq \Gamma_n^*$ such that $Q_i^g = Q$.
Moreover, $N_\Gamma(Q_i)^g = N_\Gamma(Q)$. As $\tau$ fixes $g$ by
\eqref{E:tauGamman}, it follows that $\tau$ is the identity on $N_\Gamma(Q)$ by
Lemma~\ref{L:1cocyclesetup}(b). We are therefore reduced to proving
\eqref{E:tau=1onNQi}.

Now for $i = n$, we have $N_\Gamma(Q_n) \subseteq \Gamma_n^*$, and so
$\tau$ is the identity on this normalizer by \eqref{E:tauGamman}. Hence $n \geq
2$ since $\Gamma$ is a counterexample.  

Let $Q \in \W$ be a well-placed subgroup conjugate to $Q_{n-1}$. Then $Q$ is
conjugate to $Q_{n-1}$ by an element $g \in N_\Gamma(Q_n)$.  By
Lemma~\ref{L:1cocyclesetup2}(a), the restriction of $\tau$ to $N_\Gamma(Q)$ is
conjugation by an element of $V$ that centralizes $N_\Gamma(Q) \cap
N_\Gamma(Q_n)$. As $\tau$ fixes $g$, similarly $\tau$ acts by conjugation on
$N_\Gamma(Q_{n-1})$ by an element $z$ of $C_V(N_\Gamma(Q_{n-1}) \cap
N_\Gamma(Q_n))$.  By \eqref{E:N_Gamma(Qi)}, $z \in \gen{z_{n-1}, z'_{n-1}}$ and
$z_n = z_{n-1}z'_{n-1}$.  As $z_{n-1}$ centralizes $N_\Gamma(Q_{n-1})$, if
necessary we may replace $t$ by $t' = t\,du$ where $u(P) = z_n$ for each $P \in
\Q$ and obtain that $\tau$ is the identity on $N_\Gamma(Q_{n-1})$, because the
rigid map associated with $du$ is conjugation by $z_n$ on $N_\Gamma(Q_{n-1})$
by Lemma~\ref{L:1cocyclesetup2}(b). Also, $\tau$ remains the identity on
$\Gamma^*_n$ after this adjustment, because $z_n \in C_V(\Gamma^*_n)$.  Thus,
$n \geq 3$ since $\Gamma$ is a counterexample. 


Finally, fix $i$ with $1 \leq i \leq n-2$.  Since $\bar{Q}_i =
\gen{(1,2),\dots,(2i-1,2i)}$, we see that $N_G(\bar{Q}_i) = G_i \times G_i'$
where $G_i \cong C_2 \wr S_i$ moves the first $2i$ points, and where $G_i'
\cong S_{m-2i}$ moves the remaining points in the natural action.  Let
$\Gamma_i$ and $\Gamma_i'$ be the preimages of $G_i$ and $G_i'$ in $\Gamma$, so
that $N_\Gamma(Q_i) \subseteq \Gamma_i\Gamma_i'$.  By Lemma~\ref{L:symgen},
$G_i'$ is generated by $G_i' \cap N_G(\bar{Q}_n)$ and $G_i' \cap
N_G(\bar{Q}_{n-1})$, and so $\Gamma_i'$ is the subgroup generated by
$C_\Gamma(D)$, $N_{\Gamma_i'}(Q_n)$, and $N_{\Gamma_i'}(Q_{n-1})$ by
Lemma~\ref{L:normquot}. Hence $\tau$ is the identity on $\Gamma_i'$.  As
$\Gamma_i \subseteq \Gamma_n^*$, $\tau$ is also the identity on $\Gamma_i$ by
\eqref{E:tauGamman}. Therefore, $\tau$ is the identity on $N_\Gamma(Q_i)$.
This completes the proof of \eqref{E:tau=1onNQi}.  We conclude that $\tau$ is
the identity on $\Gamma^*$, and now Lemma~\ref{L:1cocyclesetup}(c) shows that
this is contrary to our choice of $t$.  
\end{proof}

\appendix
\section{Modified norm argument}\label{A:norm2}

Here we give a proof of Theorem~\ref{T:glawc2}, which is the main technical
result needed for the results of \S\ref{S:norm2}. We have postponed the proof
until now so as to not interrupt the flow of that section, and because the
proof is similar to that given in Theorem~A1.4 of \cite{Glauberman1971}.

Given a group $G$ and two nonempty subsets $X$ and $Y$ of $G$, define the
product set 
\[
X \cdot Y = \{xy \mid x \in X, y \in Y\}. 
\]

\begin{definition}\label{D:settransversalnorm}
Let $G$ be a finite group and $V$ an abelian group on which $G$ acts. Let $X$
be a subset of $G$. 
\begin{enumerate}
\item[(a)] A subset $Y$ of $G$ is a \emph{transversal} to $X$ in $G$ if for each $g
\in G$, there are unique $x \in X$ and $y \in Y$ such that $g = xy$. 
\item[(b)] The \emph{norm} from $X$ to $G$ relative to the transversal $Y$ is the
group homomorphism $\N_{X;\,Y}^G\colon C_V(X) \to C_V(G)$ given by $v \mapsto
\prod_{y \in Y}v^y$. 
\end{enumerate}
\end{definition}

Given a subset $X$, a transversal $Y$ to $X$ in $G$, and an element $g \in G$,
one sees that the map $y \mapsto y_g$ is a bijection $Y \to Y$, where $y_g \in
Y$ is the unique element such that $yg = xy_g$ for some $x \in X$. Hence the
image of $\N_{X;\,Y}^G$ does indeed lie in $C_V(G)$. 

\begin{lemma}\label{L:normproddecomp}
Let $P$ be a finite $p$-group, let $V$ an abelian group on which $P$ acts, and
let $Q$ and $R$ be subgroups of $P$. Then there exists a transversal to $Q
\cdot R$ in $P$, and $\N^P_{R} = \N^P_{Q \cdot R;\, Y}\N^Q_{Q \cap
R}|_{\,C_V(R)}$ for any such transversal $Y$. 
\end{lemma}
\begin{proof}
This is a combination of Lemmas~A1.1 and A1.2 in \cite{Glauberman1971}, with
the statement on norms following from Lemma~A1.1(a) there and
Definition~\ref{D:settransversalnorm}(b) here.  
\end{proof}

\begin{theorem}\label{T:norm2-appendix}
Suppose $G$ is a finite group, $S$ is a Sylow $p$-subgroup of $G$, and $D$ is
an abelian $p$-group on which $G$ acts. Let $\A$ be a nonempty set of subgroups
of $S$, and set $J = \gen{\A}$. Let $H$ be a subgroup of $G$ containing
$N_G(J)$, and set $V = \Omega_1(D)$. Assume that $J$ is weakly closed in $S$
with respect to $G$, and that
\begin{align}\label{E:normcond2app}
\parbox[t]{0.85\linewidth}{whenever $A \in \A$, $g \in G$, and $A \nleq H^g$,
then $\N_{A \cap H^g}^{A} = 1$ on $V$,}
\end{align}
or more generally,
\begin{align}\label{E:gennormcond2app}
\parbox[t]{0.85\linewidth}{whenever $g \in G$ and $J \nleq H^g$,
then $\N_{J \cap H^g}^{J} = 1$ on $V$.}
\end{align}
Then $C_D(H) = C_D(G)$. 
\end{theorem}
\begin{proof}

We follow the argument from \cite[Theorem~A1.4]{Glauberman1971}.  Let $H$ be a
subgroup of $G$ containing $N_G(J)$. Then $S \leq H$ since $J$ is weakly closed
in $S$ with respect to $G$.

In the situation of \eqref{E:gennormcond2app}, there is $A \in \A$ with $A
\nleq J \cap H^g$, since $J = \gen{\A}$. Then $A \cap (J \cap H^g) = A \cap
H^g$, and we see that \eqref{E:gennormcond2app} follows from
$\eqref{E:normcond2app}$ upon applying Lemma~\ref{L:normproddecomp} with $J$,
$A$, and $J \cap H^g$ in the roles of $P$, $Q$, and $R$, respectively.

Thus, we assume \eqref{E:gennormcond2app} and prove $C_D(H) = C_D(G)$ by
induction on the order of $D$. We may assume $D > 1$.  The $p$-th power
homomorphism on $D$ has kernel $V$ and image $\mho^1(D)$, and so $D/V \cong
\mho^1(D)$. Since $\mho^1(D) < D$ and $\Omega_1(\mho^1(D)) \leq V$, the pair
$(G, \mho^1(D))$ satisfies the hypotheses of the theorem in place of $(G,D)$.
Thus \begin{eqnarray}\label{E:ind} C_{D/V}(G) = C_{D/V}(H).
\end{eqnarray}
by induction.  

Let $z \in C_D(H)$ and suppose first that $\gen{V,z} < D$. The coset $Vz$ is
fixed by $H$, and so it is fixed by $G$ by \eqref{E:ind}. Thus, $\gen{V,z}$ is
$G$-invariant. Apply induction with $\gen{V,z}$ in the role of $D$ to obtain
that $z \in C_D(G)$ as required. 

Next assume that $\gen{V,z} = D$ and $V < D$. Then $C_V(H) = C_V(G)$ by
induction.  Set $z' = \N_H^G(z)$. Then $z' \in C_D(G) \leq C_D(H)$. Since $Vz$
is $G$-invariant, $z' \equiv z^{|G:H|}$ modulo $V$. Then as $|G:H|$
is prime to $p$, we see that $\gen{V,z'} = D$ and 
\[
z \in C_D(H) = C_V(H)\gen{z'} = C_V(G)\gen{z'} = C_D(G) 
\]
as required.

Finally assume that $V = D$. Given a set $[H\bs G/J]$ of $H$-$J$ double coset
representatives in $G$ containing the identity, and a transversal $[J/J\cap
H^g]$ to $J \cap H^g$ in $J$ for each $g \in [H \bs G/J]$, then the disjoint
union of $g[J/J\cap H^g]$ as $g$ ranges over $[H\bs G/J]$ is a transversal to
$H$ in $G$. Further, $\N^J_{J \cap H}(z) = \N_J^J(z) = z$.  Thus, the norm map
decomposes as
\begin{eqnarray} \label{E:mackeynorm}
\N_H^G(z) = \prod_{g \in [H\bs G/J]} \N^J_{J \cap H^g}(z^g) = z \prod_{g
\in [H\bs G/J]-\{1\}} \N_{J \cap H^g}^J(z^g). 
\end{eqnarray}
If $g \in [H\bs G/J] - \{1\}$ and $J \leq H^g$, then we may choose $h \in H$
such that $J^{g^{-1}h} \leq S$. Then $g^{-1}h \in N_G(J) \leq H$, since $J$ is
weakly closed in $S$ with respect to $G$, and so $HgJ = Hh^{-1}gJ = H$ yields
$g = 1$ by our choice, a contradiction. Thus, $J \nleq H^g$ for each $g \in
[H\bs G/J]-\{1\}$.  We conclude that $\N_{J \cap H^g}^J(z^g) = 1$ for each such
$g$ from \eqref{E:gennormcond2app}, and then $z = \N_H^G(z) \in C_V(G)$ from
\eqref{E:mackeynorm}. 
\end{proof}

\section{Conjugacy and conjugacy functors}\label{A:groups}

We give here some elementary lemmas from finite group theory that are needed at
various places in the paper. We also discuss the notion of a $\Gamma$-conjugacy
functor $W$, and describe how it gives rise to the $\Gamma$-conjugation family
of subgroups well-placed with respect to $W$, which is also a conjugation
family for the fusion system of $\Gamma$. 

\begin{lemma}[Burnside]\label{L:burnside}
Let $G$ be a finite group and $S$ a Sylow $p$-subgroup of $G$.  Assume that $J$
is an abelian subgroup of $S$ that is weakly closed in $S$ with respect to $G$
and that $X$ and $Y$ are subgroups of $J$. If $X$ and $Y$ are conjugate in $G$,
then they are conjugate in $N_G(J)$. 
\end{lemma}
\begin{proof}
Assume $X$ and $Y$ are conjugate in $G$, and fix $g \in G$ with $X^g = Y$.
Since $\gen{J, J^g} \leq C_G(Y)$, we may choose $h \in C_G(Y)$ such that
$\gen{J^h, J^g}$ is a $p$-group. Choose $g_1 \in G$ with $\gen{J^{hg_1},
J^{gg_1}} \leq S$. Then $J^{hg_1} = J = J^{gg_1}$ since $J$ is weakly closed in
$S$ with respect to $G$. Thus, $gh^{-1} \in N_G(J)$, and $X^{gh^{-1}} =
Y^{h^{-1}} = Y$ since $h^{-1}$ centralizes $Y$. 
\end{proof}

The following lemma gives an alternative, more elementary, argument for
Lemma~\ref{L:1cocyclesetup2}(a) using the norm map. One should apply it there
by taking $(N_\Gamma(Q), N_S(Q), Q, \tau)$ in the role of $(\Gamma, S, Y,
\tau)$ below.

\begin{lemma}
\label{L:gaschutz}
Let $(\Gamma, S, Y)$ be a general setup for the prime $p$ and $\tau$ an
automorphism of $\Gamma$ that centralizes $S$. Then $\tau$ is conjugation by an
element of $Z(S)$.  
\end{lemma}
\begin{proof}
Set $D = Z(Y)$ for short. Denote by $\hat{\Gamma} = \Gamma\gen{\tau}$ the
semidirect product of $\Gamma$ by $\gen{\tau}$, and set $\hat{S} = S\gen{\tau}
= S \times \gen{\tau}$ and $\hat{D} = D\gen{\tau} = D \times \gen{\tau}$. Then
$Y$ and $D$ are normal in $\hat{\Gamma}$, and $\hat{D} = C_{\hat{\Gamma}}(Y)$,
so that $\hat{D}$ is normal abelian in $\hat{\Gamma}$.

Let $k = |D|$. Now take any element $g$ of $\Gamma$, and let $z = g^{-1}g^\tau
= (\tau^{-1})^g \tau = g^{-1} \tau^{-1} g \tau$.  Since $\hat{\Gamma}/\Gamma$
is abelian and $\hat{D}$ is normal in $\hat{\Gamma}$, we have $z \in \Gamma
\cap \hat{D} = D \leq S$, so $z^\tau = z$.  By induction, $g^{\tau^i} = gz^i$
for all $i \geq 1$.  Hence $g^{\tau^k} = gz^k = g$.  Since $g$ is arbitrary and
$\tau$ is an automorphism, $\tau^k =1$.  Thus, the order of $\tau$ is a power
of $p$.

Consider the norm $\N := \N_{\hat{S}}^{\hat{\Gamma}}\colon C_{\hat{D}}(\hat{S})
\to C_{\hat{D}}(\hat{\Gamma})$, and set $n = |\Gamma:S|$. Since $\hat{\Gamma}$
centralizes $\hat{D}/D$ and since $n$ is prime to $p$, the restriction
$\N|_{\gen{\tau}}$ is injective. Choose an integer $m$ such that $mn=1$ (mod
$k$). Then $\N(\tau) \equiv \tau^n$ (mod $D$) and $\N(\tau)^m \equiv \tau$
(mod $D$). Let $\sigma = \N(\tau)^m$, and choose $d \in D$ with $\sigma = \tau
d$. Then $\tau \equiv d^{-1}$ modulo $Z(\hat{\Gamma})$ since $\hat{\Gamma}$
centralizes $\sigma$.  Further, since $\hat{S}$ centralizes $\tau$, we see that
$d \in C_{\hat{\Gamma}}(\hat{S}) \cap D = Z(S)$, as desired. 
\end{proof}

\begin{lemma}
\label{L:symgen}
Let $n \geq 2$ and let $G$ be the symmetric group $S_{2n+1}$. Set
\begin{align*}
R_1 &= \gen{(1,2),(3,4),\dots,(2n-1,2n)}, \text{ and}\\
R_2 &= \gen{(1,2),(3,4),\dots,(2n-3,2n-2)}.
\end{align*}
Then $G$ is generated by $N_G(R_1)$ and $N_G(R_2)$.
\end{lemma}
\begin{proof}
Let $H = \gen{N_G(R_1), N_G(R_2)}$ and $\Omega = \{1,2,\dots,2n+1\}$.  Now
$N_G(R_1)$ is transitive on $\Omega-\{2n+1\}$. Similarly, $N_G(R_2)$ is
transitive on $\Omega-\{2n-1,2n,2n+1\}$ and contains a subgroup inducing the
symmetric group on $\{2n-1,2n,2n+1\}$. Therefore, $H$ is transitive on $\Omega$
and the stabilizer of $2n+1$ in $H$ is transitive on $\Omega-\{2n+1\}$. Since
$H$ contains the transposition $(2n,2n+1)$ and is $2$-transitive on $\Omega$,
it contains all transpositions. Hence, $H = G$.  
\end{proof}

We next give the background on conjugacy functors and well-placed subgroups,
which are used in \S\ref{S:reduction2} and \S\ref{S:transvections}.

\begin{definition}\label{D:conjfunctor}
Let $\Gamma$ be a finite group with Sylow $p$-subgroup $S$. A
$\Gamma$-\emph{conjugacy functor on} $\sS(S)$ is a mapping $W\colon \sS(S) \to
\sS(S)$ such that for all $P \leq S$, 
\item[(a)] $W(P) \leq P$;
\item[(b)] $W(P) \neq 1$ whenever $P \neq 1$; and
\item[(c)] $W(P)^g = W(P^g)$ whenever $g \in \Gamma$ with $P^g \leq S$. 
\end{definition}

\begin{lemma}\label{L:conjfunctor}
Let $\Gamma$ be a finite group, $S$ a Sylow $p$-subgroup of $\Gamma$, and $W$ a
$\Gamma$-conjugacy functor on $\sS(S)$.  Then for all $P \leq S$, 
\begin{enumerate}
\item[(a)] $N_S(P) \leq N_S(W(P))$;
\item[(b)] $W(P) = W(N_S(W(P)))$ if and only if $W(P) = W(S)$; and 
\item[(c)] $P = N_S(W(P))$ if and only if $P = S$. 
\end{enumerate}
\end{lemma}
\begin{proof}
Let $P \leq S$ and $T = N_S(W(P))$. Part (a) holds by
Definition~\ref{D:conjfunctor}(c). If $W(T) = W(P)$, then by (a), $N_S(T) \leq
N_S(W(T)) = N_S(W(P)) = T$, so that $T = S$ and $W(P) = W(T) = W(S)$. Now (b)
holds since the converse is clear.

If $P = T$, then again $N_S(P) \leq T = P$, so that $P = S$. Now (c) holds
since the converse is clear. 
\end{proof}

A $\Gamma$-conjugacy functor $W$ on $\sS(S)$ can be uniquely extended to a
$\Gamma$-conjugacy functor $\widehat{W}$ in the sense of \cite[\S
5]{Glauberman1971}: given a $p$-subgroup $P$ of $\Gamma$, choose $g \in \Gamma$
with $P^g \leq S$ and define $\widehat{W}(P) = W(P^g)^{g^{-1}}$. Then
$\widehat{W}$ is a mapping on all $p$-subgroups of $\Gamma$ that is uniquely
determined, by Lemma~\ref{D:conjfunctor}(c). 

Each $\Gamma$-conjugacy functor $W$ gives rise to a conjugation family via its
well-placed subgroups.  A \emph{conjugation family} for the fusion system $\F$
over $S$ is a collection $\C$ of subgroups of $S$ such that every morphism in
$\F$ is a composition of restrictions of $\F$-automorphisms of the members of
$\C$. A conjugation family $\C$ for $S$ in $\Gamma$ in the sense of \cite[\S
3]{Glauberman1971} is itself a conjugation family for $\F_S(\Gamma)$ in the
above sense.

For a $\Gamma$-conjugacy functor $W$ on $\sS(S)$ and a subgroup $P \leq S$,
define $W_1(P) = P$ and, for all $i \geq 2$, define inductively $W_{i}(P) =
W(N_S(W_{i-1}(P)))$. Then $P$ is said to be \emph{well-placed} (with respect to
$W$) if $W_i(P)$ is fully $\F_S(\Gamma)$-normalized for all $i \geq 1$. 

\begin{theorem}\label{T:wellplaced}
Let $\Gamma$ be a finite group, $S$ a Sylow $p$-subgroup of $\Gamma$, and $W$ a
$\Gamma$-conjugacy functor on $\sS(S)$. Then every subgroup of $S$ is
$\Gamma$-conjugate to a well-placed subgroup of $S$. The set of well-placed
subgroups of $S$ forms a conjugation family for $\F_S(\Gamma)$. 
\end{theorem}
\begin{proof}
This is a combination of Lemma~5.2 and Theorem~5.3 of \cite{Glauberman1971},
given the above remarks.  
\end{proof}

\bibliographystyle{amsalpha}{ }
\bibliography{/home/justin/owncloud/math/research/mybib}
\end{document}